\documentclass[10pt,a4paper]{article}

\title{Parabolic eigenvarieties via overconvergent cohomology}
\author{Daniel Barrera Salazar and Chris Williams}
\date{}

\newcommand{\Addressesshort}{{
		\footnotesize
		\setlength{\parindent}{0pt}
		Daniel Barrera Salazar; Universidad de Santiago de Chile $\ \cdot\ $
		\texttt{daniel.barrera.s@usach.cl}
		
		Chris Williams; University of Warwick $ \ \cdot \ $
		\texttt{christopher.d.williams@warwick.ac.uk}	
}}

\usepackage[margin=1.12in]{geometry}
\usepackage{comment}
\renewcommand{\baselinestretch}{1.09}

\makeatletter 
\def\input@path{{../}} 
\makeatother
\usepackage{preamble}
\usepackage{chngcntr}
\counterwithin{table}{section}
\newcommand{\sar}[2]{\ar@{}[#1]|-*[@]{#2}}
\begin{document}
\maketitle

\begin{abstract}
	Let $\cG$ be a connected reductive group over $\Q$ such that $G = \cG/\Qp$ is quasi-split, and let $Q \subset G$ be a parabolic subgroup. We introduce parahoric overconvergent cohomology groups with respect to $Q$, and prove a classicality theorem showing that the small slope parts of these groups coincide with those of classical cohomology. This allows the use of overconvergent cohomology at parahoric, rather than Iwahoric, level, and provides flexible lifting theorems that appear to be particularly well-adapted to arithmetic applications. When $Q$ is a Borel, we recover the usual theory of overconvergent cohomology, and our classicality theorem gives a stronger slope bound than in the existing literature. We use our theory to construct $Q$-parabolic eigenvarieties, which parametrise $p$-adic families of systems of Hecke eigenvalues that are finite slope at $Q$, but that allow infinite slope away from $Q$.
\end{abstract}

\setcounter{tocdepth}{2}
\footnotesize
\tableofcontents \normalsize
\setlength{\parskip}{3pt}

\section{Introduction}

\subsection{Context}
Hida and Coleman families describe the variation of automorphic representations as their weight varies $p$-adic analytically. They have become ubiquitous in many areas of number theory, and are vital tools in the study of the Langlands program and the Bloch--Kato conjectures. Their behaviour is captured geometrically in the theory of \emph{eigenvarieties}. To construct and study an eigenvariety, one requires:
\begin{itemize}\setlength{\itemsep}{0pt}
	\item a rigid analytic \emph{weight space} $\cW$, encoding $p$-adic analytic variation of weights;
	\item for each $\lambda \in \cW$, a space $M_\lambda$ that varies analytically in $\lambda$, and which carries an action of a suitable Hecke algebra;
	\item and a notion of `classical structure/classicality', relating finite-slope systems of Hecke eigenvalues appearing in $M_\lambda$ to those arising from $p$-refinements of automorphic representations of weight $\lambda$.
\end{itemize}
The eigenvariety is then a rigid analytic space $\cE$, with a weight map $w : \cE\to \cW$, whose points lying above a weight $\lambda$ parametrise finite-slope systems of Hecke eigenvalues that appear in $M_\lambda$. Via the classical structure these relate to eigensystems attached to automorphic representations.

Let $\cG$ be a connected reductive group over $\Q$, and suppose $G \defeq \cG_{/\Qp}$ is quasi-split. In this case Hansen \cite{Han17} has constructed eigenvarieties for $\cG$ by taking $M_\lambda$ to be \emph{overconvergent cohomology groups}; his work generalises earlier constructions of Ash--Stevens and Urban \cite{AS08,Urb11}. Cohomological automorphic representations of $\cG(\A)$ of weight $\lambda$ arise in the cohomology of locally symmetric spaces $S_K$ for $\cG$, of level $K$, with coefficients in an algebraic representation $V_\lambda^\vee$ of weight $\lambda$. Overconvergent cohomology is defined by replacing $V_\lambda^\vee$ with an (infinite-dimensional) module $\cD_{\lambda}^G$ of $p$-adic distributions. The classical structure is then furnished by a \emph{classicality theorem}, which says that the `non-critical/small slope' parts of the overconvergent and classical cohomology coincide, so that non-critical slope systems of Hecke eigenvalues in $M_\lambda$ are classical. Here the \emph{slope} of an eigensystem is the $p$-adic valuation of the $U$ eigenvalue (for an appropriate `controlling operator' $U$). A slope 0 eigensystem is \emph{ordinary}. 

This classicality theorem was first introduced in \cite{Ste94} for modular forms, and is a cohomological analogue of Coleman's classicality criterion \cite{Col96}. It has, in its own right, had far-reaching arithmetic consequences: to give a brief flavour, it has been used to construct $p$-adic $L$-functions \cite{PS11}, to study $\mathcal{L}$-invariants \cite{GS93}, to construct Stark--Heegner points \cite{Dar01}, and to give conjectural analogues of class field theory over real quadratic fields \cite{DV18}.

\subsection{Parabolic families and classicality} 
In the usual theory, $p$-adic families for $\cG$ encode variation with respect to a Borel subgroup $B \subset G$. In particular, $U$ is a $B$-controlling operator in the sense of \S\ref{sec:controlling}, the natural generalisation of the $U_p$ operator for modular forms. Then the eigenvariety encodes $U$-finite-slope eigensystems, and the non-critical slope bound depends on $U$. 

All of the above is defined using the Iwahori subgroup at $p$. When applying this to the study of an automorphic representation $\pi$, this forces one to work at Iwahoric level, studying `full' $p$-refinements of $\pi$. In practice, however, it is frequently more natural to work only at \emph{parahoric} level for a parabolic subgroup $Q \subset G$, corresponding to a weaker $p$-refinement. In this setting, passing further to full Iwahoric level often requires stronger hypotheses and a loss of information.  

In this paper, we present a refined version of  overconvergent cohomology which applies to $Q$-parahoric level, and prove a classicality theorem for this refined theory. We vary this in $p$-adic families and use it to construct `parabolic eigenvarieties', parametrising parabolic families of automorphic representations. This approach brings two further benefits:
\begin{itemize}\setlength{\itemsep}{0pt}
	\item the criterion for non-critical slope is weaker, giving more control in the classicality theorem;
	\item the resulting parabolic families parametrise $Q$-finite-slope eigensystems, without requiring finite slope away from $Q$.
\end{itemize}
This is offset by the fact that these spaces vary over smaller-dimensional weight spaces.

A very special case of this is as follows. Suppose $F$ is a real quadratic field in which $p$ splits as $\pri\overline{\pri}$, and let $\cG = \mathrm{Res}_{F/\Q} \GL_2$. Then $G = \GL_2 \times \GL_2$, and $U_p = U_{\pri}U_{\overline{\pri}}$ is a $B$-controlling operator. Let $E/F$ be a modular elliptic curve with good ordinary reduction at $\pri$ and bad (additive) reduction at $\overline{\pri}$. The attached system of Hecke eigenvalues has infinite slope for $U_{\overline{\pri}}$ and hence $U_p$, and does not appear in the (2-dimensional) Hilbert eigenvariety. However, we may take a parabolic $Q = B_2 \times \GL_2 \subset G$, where $B_2$ is the Borel in $\GL_2$; then $U_{\pri}$ is a $Q$-controlling operator, and the ordinary $\pri$-refinement of $E$ satisfies the $Q$-classicality theorem, giving a 1-dimensional `$\pri$-adic family' through $E$. Moreover, this classicality yields a class in the $\pri$-adic overconvergent cohomology attached to $E$, which has been used to construct $p$-adic points on $E$ \cite{GMS15}.

\subsection{Methods and results}
Our parahoric overconvergent cohomology groups are defined using \emph{parahoric\footnote{Though we only consider parahoric subgroups attached to parabolics, we write `parahoric distributions/overconvergent cohomology' to avoid conflict with the established definition of parabolic cohomology.} distribution modules}. Any weight $\lambda$ is naturally a character on the torus $T(\Zp)$; we are most interested in those that are algebraic dominant, and call these \emph{classical}. The typical coefficient modules used in overconvergent cohomology are:
\begin{itemize}\setlength{\itemsep}{0pt}
	\item overconvergent coefficients $\cD_\lambda^G$, dual to the \emph{locally analytic} induction of $\lambda$ to the Iwahori subgroup of $G(\Zp)$,
	\item and classical coefficients $V_\lambda^\vee$, dual to the \emph{algebraic} induction of $\lambda$ to $G(\Zp)$.
\end{itemize} 
We consider a hybrid construction, defining spaces $\D_\lambda^Q$ by taking the algebraic induction of $\lambda$ to the Levi subgroup $L_Q$ of $Q$, then (locally) analytically inducing to the parahoric subgroup for $Q$, then taking the dual. These groups are naturally quotients of $\cD_\lambda^G$. Moreover if we take $Q = B$ to be the Borel, we recover $\cD_{\lambda}^G$; and if we take $Q=G$ the `trivial' parabolic we recover $V_{\lambda}^\vee$. All of this is described in \S\ref{sec:parabolic distributions}, and summarised in Table \ref{table of coeffs}.

In \S\ref{sec:classicality}, we construct a parahoric version of Jones--Urban's locally analytic BGG resolution. This is an analytic version of the main result of \cite{Lep77}, and provides a tool for our main result, which is a $Q$-classicality theorem giving an isomorphism between the small-slope parts of cohomology with $\D_{\lambda}^Q$ and $V_\lambda^\vee$ coefficients.  In particular, in Theorem \ref{thm:Q-classicality} we prove:

\begin{theorem-intro}\label{thm:intro}
	Let $Q = P_0 \subset P_{1} \subset \cdots \subset P_m = G$ be a maximal chain of parabolics containing $Q$, and let $U_Q$ be a $Q$-controlling operator which factorises as $U_Q = U_1 \cdots U_m$, where each $U_i\cdots U_m$ is a $P_{i-1}$-controlling operator. Let $\phi$ be a system of Hecke eigenvalues and $\lambda$ a classical weight.
	
There exist precise bounds $h_i \in \Q_{>0}$, depending on $\lambda$, such that if $v_p(\phi(U_i)) < h_i$ for each $i$, then the $\phi$-parts of the weight $\lambda$ classical and $Q$-overconvergent cohomology are isomorphic.
\end{theorem-intro}

If $v_p(\phi(U_i)) < h_i$ for each $i$, we say $\phi$ has \emph{$Q$-non-critical slope}. The notion of being a controlling operator, and the precise values of $h_i$, are described in terms of root data and Weyl groups, which we recap in \S\ref{sec:root data}. We describe this theorem in a number of explicit cases in Examples \ref{ex:classicality}.  

\begin{remark*}
If $Q$ is the Borel, the most general classicality theorems for (Iwahoric) overconvergent cohomology that currently appear in the literature -- for example, \cite[Prop.\ 4.3.10]{Urb11} -- require $v_p(\phi(U_Q)) < \mathrm{min}_i(h_i)$, so even in this case we give a significant improvement on the known range of non-critical slopes. Such improved ranges were known to exist in other settings (for example, see \cite[\S4.4]{EmertonJacquetI}), and we believe an analogue for overconvergent cohomology was expected by experts. However, it does not appear in the literature, which we aim to rectify here.
\end{remark*}

The parahoric overconvergent cohomology groups can be naturally varied analytically in the weight, from which the construction of $p$-adic families and eigenvarieties -- and their basic properties -- is fairly standard. In particular, we construct rigid analytic spaces whose points parametrise $Q$-finite slope systems of eigenvalues, and coherent sheaves on these spaces that interpolate $Q$-finite slope eigenspaces in classical cohomology. We describe this in \S\ref{sec:parabolic eigenvarieties}. In \S\ref{sec:cuspidal families}, we give sufficient conditions for the existence of parabolic families of cuspidal automorphic representations.

\subsection{A note on assumptions}
We will use results from \cite{Urb11} and \cite{Han17}, which work in slightly different settings to us. In \cite{Urb11}, Urban's main applications are in the case where $\cG$ is quasi-split at $p$ and satisfies the Harish-Chandra condition at infinity (i.e. $\cG^{\mathrm{der}}(\R)$ admits discrete series). The Harish-Chandra condition is assumed only to control the geometry of the eigenvarieties he constructs. In particular it is not used anywhere in \S2,3 of \cite{Urb11}, which covers the results we use; in these sections Urban sets up the theory of (Iwahoric) overconvergent cohomology assuming only $\cG_{/\Qp}$ is quasi-split. (We indicate briefly where the Harish-Chandra condition is useful in our setting. In the notation of Definition \ref{def:defect} below, it implies that at any $Q$-non-critical slope cuspidal point $x$ we have $\ell_Q(x) = 0$; and thus by Proposition \ref{prop:dimension}, any irreducible component of the parabolic eigenvariety through $x$ has the same dimension as the weight space. Without the Harish-Chandra condition this might not be true).

In \cite{Han17}, Hansen works under the assumption that $\cG_{/\Qp}$ is split instead of quasi-split. This appears to have been done only for convenience, since (as explained in \cite[\S3.1.1,\S3.2]{Urb11}) the formalism of locally analytic distributions goes through equally well when $\cG$ is quasi-split, up to keeping track of a finite field extension (the field $L$ for us). Moreover Hansen's main tools -- the spectral sequences -- require only formal properties of distributions that hold in the quasi-split case.

In fact, as remarked on p.1712, footnote 16 of \cite{Urb11}, it should be possible to drop the quasi-split requirement altogether if one uses Bruhat--Tits buildings. One then replaces the parahoric subgroup with any open compact subgroup with a Bruhat--Iwahori decomposition. This approach is taken in \cite{Loe11} and \cite{HL11}, where there are no assumptions at all on $\cG$ at $p$. We have opted to stick to the notationally much simpler, but still very general, quasi-split setting.

Finally, we choose to use compactly supported cohomology throughout this paper as it best suits our future applications, but all of the results go through identically replacing this with singular cohomology (and, in \S\ref{sec:parabolic eigenvarieties}, Borel--Moore homology with singular homology).

\subsection{Comparison to the literature}\label{sec:comparison}

Constructions of parabolic families/eigenvarieties have been previously given using methods different to this paper. The theory was introduced for Hida families in \cite{HidP-ord}, and other papers on this subject include \cite{Loe11} (for unitary groups), \cite{Pil12} (Hida theory for Siegel modular forms), and in particular \cite{HL11}, which treats a very general setting using Emerton's completed cohomology. They are also related to the $\mu$-ordinary setting of \cite{EM19}. Parabolic families have important applications in arithmetic: for example, in the case of $\cG = \mathrm{GSp}_4$, Siegel-parabolic families are used in \cite[\S17]{LZ20}, where new cases of the Bloch--Kato conjecture are proved; when $\cG$ is a definite unitary group, parabolic eigenvarieties were used in \cite{Che20} to attach Galois representations to certain regular, polarised automorphic representations of $\mathrm{GL}_n$; and parabolic Hida families are used in upcoming work of Caraiani--Newton to answer deep questions about local--global compatibility for Galois representations.

In this spirit, the main motivation for giving a \emph{new} version of this theory comes through arithmetic applications, for which parahoric overconvergent cohomology appears particularly well-suited; it adapts a very useful arithmetic tool (overconvergent cohomology) to a setting of increasing arithmetic interest (parahoric level/families).

 This utility is illustrated in the example of $\GL_2$ over a number field $F$, where special cases of the above theory have appeared repeatedly:

\begin{itemize}\setlength{\itemsep}{0pt}
	\item[--]  In the case where $F$ is totally real, partial $\pri$-adic families were used in \cite{BDJ17} and \cite{JoNewParity}, with applications to the trivial zero and parity conjectures respectively. 
	\item[--] For more general $F$, versions of Theorem \ref{thm:intro} have been proved and used to construct Stark--Heegner points on elliptic curves \cite{Tri06, GM14, GMS15}, and when $F$ is imaginary quadratic, to construct conjectural Stark--Heegner cycles attached to Bianchi modular forms \cite{VW19}. It was also used in \cite{BW17} to construct $\pri$-adic $\mathcal{L}$-invariants and prove an exceptional zero conjecture for Bianchi modular forms. 
	\item[--] Moreover, versions of the refined slope conditions given by Theorem \ref{thm:intro} were used in \cite{Wil17} and \cite{BW_CJM} to construct $p$-adic $L$-functions attached to automorphic forms for $\GL_2$.
\end{itemize}
In forthcoming work with Dimitrov, we use Theorem \ref{thm:intro} in the setting of $\GL_{2n}$ over totally real fields, using the parabolic $Q$ with Levi $\GL_n\times\GL_n$, to construct $p$-adic $L$-functions attached to $Q$-non-critical conjugate-symplectic automorphic representations of $\GL_{2n}$. We use the results of the present paper to give stronger non-critical-slope and growth conditions than could be achieved with Iwahoric overconvergent cohomology. We also vary this construction in $Q$-families.

These methods also appear well-adapted to the study of the general automorphic $\mathcal{L}$-invariants defined in \cite{Geh20}, in which parabolic subgroups arise very naturally. In addition to the examples for $\GL_2$ above, a combination of parahoric overconvergent cohomology with recent work of Gehrmann and Rosso \cite{GR20} should, in nice examples (such as the setting of conjugate-symplectic $\GL_{2n}$) yield arithmetic interpretations of automorphic $\mathcal{L}$-invariants. For $\GL_2$, such interpretations are already crucial in the construction of the Stark--Heegner points/cycles mentioned above.

Finally, we note the recent related work of Loeffler \cite{Loe20} on universal deformation spaces, which can be described as `big' parabolic eigenvarieties. The eigenvarieties we construct are the `small' automorphic eigenvarieties of \S6.2 \emph{op.\ cit}.; as yet there is no `big' automorphic analogue.

\subsection{Acknowledgements}
We are very grateful to Mladen Dimitrov, who helped us work out these definitions explicitly for $\GL_{2n}$, and to David Loeffler, who gave valuable comments and suggestions on an earlier draft. We are also indebted to the referee for their careful reading of the paper, and for their valuable comments and corrections. D.B.S.\ was supported by the FONDECYT PAI 77180007. C.W.\ was funded by an EPSRC Postdoctoral Fellowship EP/T001615/1.

\section{Preliminaries and structure theory}
\label{sec:root data}
\subsection{Global notation}
Let $F$ be a number field, and for each non-archimedean place $v$ let $F_v$ denote its completion at $v$, with ring of integers $\cO_v$ and uniformiser $\varpi_v$. Let $\cG'$ be a connected reductive group over $F$, and $\cG \defeq \mathrm{Res}_{F/\Q} \cG'$ be the Weil restriction of scalars. 
We will be fundamentally interested in the cohomology of locally symmetric spaces attached to $\cG$. Let $K \subset \cG(\A_f)$ be an open compact subgroup, where $\A_f$ denotes the finite adeles of $\Q$, let $C_\infty$ (resp.\ $Z_\infty$) be the maximal compact subgroup (resp.\ centre) of $\cG(\R)$, and let $K_\infty = C_\infty Z_\infty$. Then let
 \[
 	S_K \defeq \cG(\Q)\backslash \cG(\A)/KK_\infty^\circ
 \]
be the locally symmetric space attached to $K$, where $K_\infty^\circ$ is the identity component of $K_\infty$. If $M$ is a right $K$-module such that the centre $Z(K\cap \cG(\Q))$ acts trivially, then we get an associated local system on $S_K$ given by the fibres of the projection
\begin{equation}\label{eq:local systems}
	\cG(\Q) \ \backslash [\cG(\A) \times M]/ \ KK_\infty^\circ \to S_K,
\end{equation}
with action $\gamma(g,m)uk = (\gamma guk, m|u)$.

\subsection{Local notation and root data at $p$}\label{sec:local notation}
Let $G = \cG_{/\Qp}$. We assume that $G$ is quasi-split, and splits over a (fixed) finite Galois extension $E/\Qp$. As far as possible we will suppress $E$ from notation. We take $\cG'/F_v$ and $G$ to have (henceforth fixed) models over $\cO_v$ and $\Zp$ respectively. Let $T$ be a maximal torus in $G$, and $B$ a Borel subgroup containing $T$. Let $B^-$ denote the opposite Borel, and $N, N^-$ the unipotent radicals of $B, B^-$. Attached to all of these groups we have corresponding Lie algebras $\fg, \ft,\fb, \fb^-, \fn, \fn^-$ over $\Qp$. Let 
\[  
	X^\bullet(T) \defeq \Hom(T,\ \bG_{m}), \hspace{12pt} X_\bullet(T) \defeq \Hom(\bG_{m},\ T)
\]
be the lattices of algebraic characters and cocharacters of the torus, and $\langle\ ,\ \rangle$ the canonical pairing on $X^\bullet(T) \otimes X_\bullet(T)$. Let $R \subset X^\bullet(T)$ denote the set of roots for $(G,T)$. For each root $\alpha$, let $H_\alpha \in \ft$ and $\alpha^\vee \in X_\bullet(T)$ be the corresponding coroots, defined so that $\langle \alpha, \alpha^\vee\rangle = \alpha(H_\alpha) = 2$. We fix a basis $X_\alpha$ of 
\[
\fg_\alpha \defeq \{X \in \fg : \mathrm{ad}(t) \cdot X = \alpha(t)X \ \text{ for all }t \in T\}
\]
normalised so that $[X_\alpha,X_{-\alpha}] = H_\alpha$ in $\fg$. Our choice of Borel fixes a set of positive roots $R^+ \subset R$ and a set $\Delta \subset R^+$ of simple roots. We say a character $\lambda \in X^\bullet(T)$ is \emph{dominant} (with respect to $B$) if $\langle \lambda, \alpha\rangle \geq 0$ for all $\alpha \in \Delta$.

Let $W_G$ denote the Weyl group of $(G,T)$, generated by reflections $w_\alpha$ for $\alpha \in \Delta$, acting on $X^\bullet(T)$ by $\lambda^{w_\alpha} = \lambda - \lambda(H_\alpha)\alpha.$ Also define the $*$-action of $W_G$ on $X^\bullet(T)$ by 
\[
	w * \lambda = (\lambda + \rho)^w - \rho, \hspace{20pt} \lambda \in X^\bullet(T),\  w \in W_G,
\]
where $\rho = \tfrac{1}{2} \sum_{\alpha \in R^+} \alpha \in X^\bullet(T) \otimes_{\Z} \tfrac{1}{2}\Z$ is half the sum of the positive roots. One may check (see e.g.\ the proof of \cite[Prop.\ 3.2.11]{Urb11}) that this action is by
\begin{equation}\label{eq:weyl action on weights}
	w_{\alpha} * \lambda = \lambda - [\langle \lambda, \alpha^\vee\rangle + 1]\alpha.
\end{equation}

\begin{example}
	To anchor this general framework, we keep in mind the familiar example of $\GL_n/\Q$. Here $G$ is split, $\fg = M_n(\Q)$, $X^\bullet(T) = \Z e_1 \oplus \cdots \oplus \Z e_n$, and $X_\bullet(T) = \Z e_1^\vee \oplus \cdots \oplus \Z e_n^\vee$. For $B$ the upper-triangular Borel, $\Delta = \{\alpha_1 = e_1 - e_2, ..., \alpha_{n-1}= e_{n-1}-e_n\}$. We have $\alpha_i^\vee = e_i^\vee - e_{i+1}^\vee,$ $H_{\alpha_i}$ is the $n \times n$ matrix with $(i,i)$ entry 1, $(i+1,i+1)$ entry $-1$ and all other entries 0, and $X_{\alpha_i}$ is the $n \times n$ matrix with $(i,i+1)$ entry 1 and all others 0. The Weyl group is $S_n$; the standard action is by permutations of the $e_i$, and the $*$ action on $\lambda = (\lambda_1,...,\lambda_n) = \lambda_1 e_1 + \cdots + \lambda_n e_n$ is
	\begin{equation}
		w_{\alpha_i} * \lambda = (\lambda_1,\ ...,\  \lambda_{i-1},\ \lambda_{i+1}-1,\ \lambda_i + 1,\ \lambda_{i+2}, \ ...,\  \lambda_n).
	\end{equation}
	The dominant weights are the $\lambda$ with $\lambda_m \geq \lambda_{m+1}$ for all $m$. In particular, if $\lambda$ is dominant, then $w_{\alpha_i}*\lambda$ is \emph{never} dominant for any $i$, as $\lambda_{i+1} - 1 < \lambda_i + 1.$
\end{example}

\subsection{Parabolic subgroups}
There is a well-known correspondence between the standard parabolic subgroups $B \subset Q \subset G$ and subsets of the simple roots: if $\fq \defeq \mathrm{Lie}(Q)$, we let
\begin{equation}\label{eq:delta Q}
\Delta_Q \defeq \{\alpha \in \Delta: X_{-\alpha} \in \fq\}.
\end{equation}
The correspondence $Q \leftrightarrow \Delta_Q$ is inclusion-preserving: in particular, $\Delta_B = \varnothing$ and the maximal standard parabolics correspond to excluding a single simple root. It is convenient (if non-standard) to allow $G$ to be the `trivial' parabolic subgroup, equal to its Levi subgroup and with $\Delta_G = \Delta$.

Let $L_Q$ denote the Levi group attached to $Q$, and $N_Q$ the unipotent radical of $Q$, so that $Q = L_Q N_Q$. Note $\Delta_Q$ can be identified with $\Delta_{L_Q}$.  Also let $Q^-$ and $N_Q^-$ be the opposite groups. 

Define the \emph{parahoric subgroup at $Q$} to be $J_Q = \prod_{v|p} J_{Q,v}$, where 
\[
J_{Q,v} \defeq \{g \in \cG'(\cO_v): g \newmod{\varpi_v} \in Q(\F_p)\}.
\]	
We also define $J_Q^- \defeq J_Q \cap N^-_Q(\Zp)$. For non-trivial $Q$ we have a parahoric decomposition
\begin{equation}\label{eq:parahoric decomp}
J_Q = J_Q^- \cdot L_Q(\Zp) \cdot N_Q(\Zp),
\end{equation}
and for $g \in J_Q$, we write this as $g = n_g^- \cdot t_g \cdot n_g$. If the context is clear, we sometimes drop the subscript $g$. Note that when $Q = B$ is the Borel, $J_B$ is the usual Iwahori subgroup and we recover the Iwahori decomposition \cite[Prop.\ 5.3.3]{Mat77}.

\subsection{The Hecke algebra}
Fix a parabolic subgroup $Q$, and let $K = \prod_{v\nmid \infty}K_v \subset \cG(\A_f)$ be an open compact subgroup. We take $K$ to be \emph{parahoric} in that $K_p \defeq \prod_{v|p}K_v \subset J_Q \subset G(\Zp)$. To define the ($Q$-parahoric) Hecke algebra at $p$, we define
\begin{equation}\label{eq:TQ+}
T^+ \defeq \{t \in T(\Qp): t^{-1}\cdot N(\Zp)\cdot t \subset N(\Zp)\}.
\end{equation}

\begin{proposition}\label{prop:valuation T++} 
	\begin{itemize}\setlength{\itemsep}{0pt}
		\item[(i)] An element $t \in T(\Qp)$ is in $T^+$ if and only if $v_p(\alpha(t)) \leq 0$ for all $\alpha \in  \Delta$. 
		\item[(ii)] If $t \in T^+$, then $t^{-1} \cdot N_Q(\zp)\cdot t \subset N_Q(\Zp)$ for any parabolic $Q$. 
		\item[(iii)] If $t \in T^+$, then $t^{-1} \cdot B(\zp) \cdot t \subset B(\Zp)$.
	\end{itemize}
\end{proposition}
\begin{proof}
	For (i), first suppose $v_p(\alpha(t)) \leq 0$ for all $\alpha$. The Lie algebra of $N$ is $\fn = \oplus_{\beta \in R^+} \Qp X_\beta \subset \fg$, which has a basis indexed by the positive roots $R^+$. We obtain co-ordinates $\{x_\beta(n) \in \Qp : \beta \in R^+\}$ for any $n \in N(\Qp)$, with the property that for any $\beta,\beta' \in R^+$, we have
	\[ 
	x_\beta(\exp(X_{\beta'})) = \left\{\begin{array}{ll} 1 &: \beta = \beta'\\
	0 &: \beta \neq \beta'.\end{array}\right.
	\] 
	Let $B_t$ be the matrix of conjugation by $t$ in this basis; it is diagonal with value $\beta^{-1}(t)$ at $(\beta,\beta)$. By the valuation condition, we have $v_p(\beta^{-1}(t)) \geq 0$ for all $t$. Now, the subgroup $N(\Zp)$ is exactly the subspace of $n$ such that $x_\beta(n) \in \Zp$ for all $\beta$, and this is clearly preserved by $B_t$.

	Conversely, if there exists $\alpha \in \Delta$ with $v_p(\alpha(t)) > 0$, then we see that $t^{-1} \exp(X_\alpha) t \notin N(\Zp)$.
	
	To see (ii), observe that we have $\fn_Q(\Qp) = \oplus_{\beta \in R^+\backslash R^+_Q} \Qp X_\beta \subset \fn(\Zp)$, where 
	\begin{equation}\label{eq:RQ+}
	R_Q^+ \defeq \{\beta \in R^+ : \beta \text{ is a root of }L_Q\}.
	\end{equation}
	Thus $N_Q(\Zp)$ is the subgroup of $N(\Zp)$ characterised by $x_\beta(n) = 0$ for $\beta \in R_Q^+$. But this space is preserved by the action of $t \in T^+$ by the arguments above. Finally (iii) is immediate since $B(\Zp) = T(\Zp)N(\Zp)$ and $T^{+}$ commutes with $T(\Zp)$.
\end{proof}

\begin{definition}\label{def:app hecke algebra}
	\begin{itemize}\setlength{\itemsep}{0pt}
		\item We define $\uhp_p(K_p)$ to be the commutative $\Qp$-algebra generated by 
		\[
		U_t \defeq [K_p t K_p], \ \ \ \ \  t \in T^+.
		\]
		\item  For the (all but finitely many) places $v$ of $F$ at which $K_v$ is hyperspecial maximal compact and $\cG'/F_v$ is unramified, define the local Hecke algebra $\uhp_v(K_v)$ to be the commutative $\Qp$-algebra generated by the double coset operators $T_v(\gamma) := [K_v\gamma K_v]$, for $\gamma \in \cG'(F_v)$.
		\item For all other $v\nmid p\infty$, define $\uhp_v(K_v) = 1.$ 
		\item We define the \emph{Hecke algebra} to be $\uhp(K) \defeq \uhp_p(K_p) \otimes \bigotimes_{v\nmid p\infty} \uhp_v(K_v).$
	\end{itemize}
\end{definition}

If $S$ is a $\Qp$-algebra, then a \emph{system of Hecke eigenvalues over $S$} is a non-trivial algebra homomorphism $\phi : \uhp(K) \rightarrow S$. If $M$ is an $S$-module upon which $\uhp(K)$ acts $S$-linearly, then we write $M_\phi$ for the localisation of $M$ -- as a $\uhp(K)\otimes_{\Qp}S$-module -- at the ideal $\ker(\phi) \subset \uhp(K)\otimes_{\Qp}S$. If $S$ is a field and $M$ a finite-dimensional $S$-vector space, this is the generalised eigenspace where $\uhp(K)$ acts as $\phi$. We say $\phi$ \emph{occurs in $M$} if $M_\phi \neq 0$.

\begin{remark*}
	We could take other choices of ramified Hecke algebra, altering the local geometry of the eigenvariety to suit particular arithmetic applications. The construction and results we present here go through for any reasonable choice of ramified Hecke algebra.
\end{remark*}

\subsection{Controlling operators}\label{sec:controlling}
In the general theory, the role of $U_p$ operator for modular forms is played by \emph{controlling operators}. Let $Q$ be a parabolic subgroup. For $s \geq 0$, let $B^s(\Zp) = \{b \in B(\Zp): b \equiv 1 \newmod{p^s}\}$ and define $N_Q^s = N_Q(\Zp)\cap B^s(\Zp)$. If $t \in T^+$, then by Proposition \ref{prop:valuation T++} we know conjugation by $t$ preserves $N_Q(\Zp)$. We define
\[
T^{++}_{Q} \defeq \big\{t \in T^+: t^{-1}\cdot N_Q^{s}\cdot t \subset N_Q^{s+1} \hspace{5pt}\forall s \geq 0\big\} = \left\{t \in T^+: \bigcap t^{-i}N_Q t^{i} = 1 \right\}.
\]
If $P \subset Q$ are two parabolics and $t \in T_P^{++}$, then $t^{-1}N_Q^st \subset N_Q(\Zp) \cap N_P^{s+1} = N_Q^{s+1}$, so $T_{P}^{++} \subset T_Q^{++}$.
\begin{proposition}\label{prop:Q controlling}
	Let $t \in T^+$. Then $t \in T^{++}_{Q}$ if and only if $v_p(\alpha(t)) < 0$ for all $\alpha \in \Delta\backslash\Delta_Q$.
\end{proposition}
\begin{proof}
	Suppose $v_p(\alpha(t)) < 0$ for all $\alpha \in \Delta\backslash\Delta_Q$, and let $n \in N_Q(\zp)$. In the notation of the proof of Proposition \ref{prop:valuation T++}, the set $R^+\backslash R_Q^+$ is precisely the set of $\beta \in R^+$ whose decomposition $\beta = \sum \alpha_i$ into simple roots (in $G$) has at least one of the $\alpha_i \in \Delta\backslash\Delta_Q$. Then $v_p(\beta(t)) < 0$ for all $\beta \in R^+\backslash R_Q^+$, and every entry of $B_t$ restricted to $N_Q(\Zp)$ is divisible by $p$. Since $N_Q^s$ is the subgroup of $n$ with $x_\beta(n) \equiv 0 \newmod{p^s}$ for all $\beta \in R^+\backslash R_Q^+$, we see that $B_t$ sends $N_Q^s$ to $N_Q^{s+1}$.
	
	Conversely, if $\alpha \in \Delta\backslash\Delta_Q$ with $v_p(\alpha(t)) = 0$, then $t^{-1}\exp(X_\alpha) t \neq I \newmod{p}$. Thus $t^{-1}\cdot N_Q^0 \cdot t \not\subset N_Q^1$, so $t \notin T^{++}_{Q}$.
\end{proof}

\begin{definition}
	If $t \in T^{++}_{Q},$ we call $U_t$ a \emph{$Q$-controlling operator}.
\end{definition}

\begin{example}
	Consider the case of $G = \GL_n$. For $Q$ the parabolic with Levi $\GL_{n-r} \times \GL_{r}$, the element $t = \mathrm{diag}(1,...,1,p,...,p)$, with $p$'s in the last $r$ entries, defines a $Q$-controlling operator, but \emph{not} a $B$-controlling operator. The element $t = \mathrm{diag}(1,p,\dots,p^{n-2},p^{n-1})$ defines a $B$-controlling operator and hence a $Q$-controlling operator for any standard parabolic $Q$.
\end{example}

\section{Parahoric overconvergent cohomology}\label{sec:parabolic distributions}
We now introduce the coefficient modules for overconvergent cohomology, using a more flexible notion of `parahoric distributions' defined relative to a parabolic $Q$. When $Q = B$ is the Borel, this specialises to the usual definition of locally analytic distributions; and when $Q = G$, we recover classical coefficient modules. Cohomology with coefficients in $Q$-parahoric distributions is more easily controlled (but varies over smaller weight spaces) as $Q$ gets larger.

%%===================================================
\subsection{Weight spaces} \label{sec:weight spaces}

Let $K \subset \cG(\bA_f)$ be an open compact subgroup such that $K_p \subset G(\Zp)$, and let $\overline{Z(K)}$ denote the $p$-adic closure of $Z_{\cG}(\Q) \cap K$ in $T(\Zp)$. 

\begin{definition}[Weights for $T$]
	Define the \emph{weight space} of level $K$ for $G$ to be the $\Qp$-rigid analytic space whose $L$-points, for $L \subset \C_p$ any sufficiently large extension of $\Qp$, are given by
		\[ 
			\cW_K(L) = \mathrm{Hom}_{\mathrm{cts}}\big(T(\Z_p)/\overline{Z(K)},L^{\times}\big).
		\]	
\end{definition}

This space has a natural group structure, and has dimension $\dim T(\Zp) - \dim \overline{Z(K)}$. It is usually more convenient to identify a weight $\lambda \in \cW_K(L)$ with the corresponding character on $T(\Zp)$ that is trivial under $\overline{Z(K)}$, and we do this freely throughout.  The condition that characters be trivial on $\overline{Z(K)}$ ensures the local systems we define later are well-defined, as discussed before \eqref{eq:local systems}. Since $K$ will typically be fixed, we will henceforth mostly drop it from the notation. 

\begin{definition}
	Each $\lambda \in X^\bullet(T)$ induces a character on $T(\Zp)$; let $X^\bullet(T)_K$ be the subspace of such $\lambda$ trivial on $\overline{Z(K)}$. There is a natural inclusion $X^\bullet(T)_{K} \subset \cW(L)$, and we call this the subspace of \emph{algebraic} weights. Via \S\ref{sec:local notation}, the algebraic weights carry the $*$-action of the Weyl group and can be paired naturally, via $\langle -,- \rangle,$ with $X_\bullet(T)$. A \emph{classical weight} is a dominant algebraic weight.
\end{definition}

When using the standard notion of distributions with respect to the Borel subgroup,  it is possible to define distributions over arbitrary affinoids in $\cW$ (see, for example, \cite[\S2.2]{Han17}). The additional flexibility we obtain with parahoric distributions, i.e.\ weaker notions of finite-slope families and non-criticality, come at the cost of less flexibility when defining distributions in families. In particular, they vary only over the following smaller weight spaces.

\begin{definition}[Weights for $Q$] Let $Q$ be a standard parabolic subgroup.
	\begin{itemize}\setlength{\itemsep}{0pt}
	\item[(i)]
		For $K$ and $L$ as above, let $\cW^Q(L)$ be the $\Qp$-rigid analytic space with $L$-points
		\[	
			\cW^Q(L) = \cW^Q_K(L) = \mathrm{Hom}_{\mathrm{cts}}\big(L_Q(\Zp)/\overline{Z(K)},L^\times\big).
		\]
	More precisely, $\cW^Q$ is the rigid generic fibre of $\mathrm{Spf}(\Zp[\![ L_Q^{\mathrm{ab}}(\Zp)/\overline{Z(K)}]\!])$, where the quotient is by the image of $\overline{Z(K)}$ in $L_Q^{\mathrm{ab}}(\Zp)$. Precomposition with $T(\Zp) \hookrightarrow L_Q(\Zp) \to L_Q^{\mathrm{ab}}(\Zp)$ realises $\cW^Q$ as a closed rigid subgroup of $\cW$. 
	\item[(ii)]
		For $\lambda_0 \in \W(\Qp)$ a fixed classical weight, define $\cW^Q_{\lambda_0}$to be the coset $\lambda_0\cW^Q$ inside $\cW$, which hence obtains the structure of a $\Qp$-rigid space. We have
		\[
			\cW_{\lambda_0}^Q(L) \defeq \{\lambda \in \cW(L) : \lambda\lambda_0^{-1} \in \cW^Q(L)\}.
		\]
	\end{itemize} 
\end{definition}

Again, we identify these weights with characters on $L_Q(\Zp)$ that are trivial under $\overline{Z(K)}$. This space has dimension $\mathrm{dim}(L_Q^{\mathrm{ab}}(\Zp)) - \mathrm{dim}(\overline{Z(K))}$, which is at most $\mathrm{dim}(\cW_K)$. Whilst we encode $\lambda_0$ in the notation, the space $\cW_{\lambda_0}^Q$ evidently only depends on $\lambda_0$ up to translation by $\cW^Q$.
\begin{example}
	Let $G = \GL_{2n}$, and $Q$ the standard parabolic with Levi $L_Q = \GL_n \times \GL_n$ embedded diagonally. Then $\cW(L)$ comprises $2n$-tuples $\lambda = (\lambda_1, ..., \lambda_{2n})$ of characters $\Zp^\times \to L^\times$ (that are trivial on $\overline{Z(K)}$), and $\cW^Q(L)$ is the subspace where $\lambda_1 = \cdots = \lambda_n$ and $\lambda_{n+1} = \cdots = \lambda_{2n}$. 
\end{example}

%%==============================================
\subsection{Parahoric distributions}
Locally analytic induction modules for a group $G$, as for example seen in \cite{AS08, Jon11, Urb11}, are usually defined through $p$-adic analytic functions on the Iwahori subgroup, and are uniquely defined by their restriction to $N(\Zp)$. For $G = \GL_n$, for example, this translates into functions that are locally analytic in $n(n-1)/2$ variables, corresponding to the off-diagonal entries in $N(\Zp)$.

We now define `partially overconvergent' distribution modules, defined with respect to the parabolic $Q$, where we only allow analytic variation in some subset of the variables in $N(\Zp)$ and dictate algebraic variation in the others. For this, we first algebraically induce up to the Levi $L_Q$, and then analytically induce to the parahoric $J_Q$. This is explained in explicit detail for $\GL_3/\Q$ in \cite[\S4.3]{Wil18}; the concrete setting \emph{op.\ cit}.\ simplifies the concepts whilst retaining the key ideas.

We recap standard results on locally analytic induction. As $G$ splits over $E$, all our coefficient modules come from representations of $\fg\subE$. For the rest of the paper, fix $L/\Qp$ finite containing $E$, and an $L$-Banach algebra $R$.

\subsubsection{Algebraic induction and highest weight representations}\label{sec:algebraic induction}
Let $\lambda \in X^\bullet(T) \subset \cW(L)$ be a classical weight for the group $G$. We have a finite-dimensional irreducible representation $V_\lambda^G$ of highest weight $\lambda$, whose $L$-points can be realised as the algebraic induction
\begin{align*}
V_{\lambda}^G(L) &\defeq \mathrm{Ind}_{B^-(\Z_p)}^{G(\Z_p)} \lambda\\ 
&\defeq \{f : G(\Z_p) \to L \ | f\text{ algebraic},\  f(n^-tg) = \lambda(t)f(g)\forall n^- \in N^-(\Zp), t \in T(\Zp), g \in G(\Zp)\}.
\end{align*}
(By an \emph{algebraic function $f : G(\Zp) \to L$}, we mean a function on $G(\Zp)$ induced by a global section $f \in L[G]$ of the structure sheaf of the (affine) group $G_{/L}$.) The space $V_{\lambda}^G(L)$ is a left $G(\Z_p)$-module by right translation, and we denote this action by $\langle \cdot \rangle_\lambda$.
Any $f \in V_\lambda^G(L)$ is determined by its restriction to the (open, dense) Iwahori subgroup $J_B$, and thus (by the transformation property and \eqref{eq:parahoric decomp}) by its restriction to $N(\Zp)$. Moreover, it is standard\footnote{See e.g.\ \cite[\S3.2.9]{Urb11}, where for $\varepsilon = 1$ this is implicit in the statement $\mathbb{V}_\lambda(\varepsilon,L) = V_\lambda(L) \cap \cA_{m}(I,L)$.}
that any algebraic $f : J_B \to L$ with $f(n^-tg) = \lambda(t)f(g)$ has a unique algebraic extension to $G(\Z_p)$.

\subsubsection{Analytic function spaces}\label{sec:analytic functions}
 Let $X \subset \Zp^r$ be open compact and $L$ and $R$ be as above. A function $f : X \to R$ is \emph{analytic} if it can be written as a convergent power series
\[
	f(x_1,...,x_r) = \sum_{n_1,...,n_r} a_{n_1,...,n_r} (x_1 - a_1)^{n_1} \cdots (x_r - a_r)^{n_r},\ \ \ \ a_{\mathbf{n}} \in R,
\]
for some $(a_1,...,a_r) \in X$. We write the space of such functions as $\cA_0(X,R)$; note that as the $a_{\mathbf{n}}$ converge to zero, $\cA_0(X,R) \cong \cA_0(X,L)\widehat{\otimes}_LR$ is the completed tensor product. We say $f$ is \emph{algebraic} if $a_{\mathbf{n}} = 0$ for all but finitely many $\mathbf{n}$, and denote the subspace of such $f$ as $V(X,R) \subset \cA_0(X,R)$. For any integer $s$, we say $f : X \to R$ is \emph{$s$-analytic} (resp.\ $s$-algebraic) if it is analytic (resp.\ algebraic) on each open disc of radius $p^{-s}$ in $X$ (inside $\Zp^r$), and write $\cA_s(X,R)$ for the space of $s$-analytic functions. Note $0$-analytic is the same as analytic, so the notation is consistent. The spaces $\cA_s(X,L)$ are Banach spaces under a suitable sup norm \cite[\S3.2.1]{Urb11}, and the inclusions $\cA_s(X,L) \subset \cA_{s+1}(X,L)$ are compact \cite[Lem.\ 3.2.2]{Urb11}. The spaces $\cA_{s}(X,R) \cong \cA_{s}(X,L)\widehat{\otimes}_LR$ inherit a Banach $R$-module structure from the completed tensor product, which can again be described in terms of sup norms (see e.g.\ \cite[\S2.2]{Han17}, \cite[Prop.\ 3.6.7]{AS08}). The inclusions $\cA_{s}(X,R) \subset \cA_{s+1}(X,R)$ are compact by \cite[Cor.\ 2.9]{Buz07}, noting the potential ONability hypothesis follows from ONability of $\cA_s(X,L)$ (see below) and Lem.\ 2.8 \emph{op.\ cit}. 
 We write $\cA(X,R) = \varinjlim_s \cA_s(X,R)$.

If $M$ is a finite Banach $R$-module, then we say a function $f : X \to M$ is $s$-analytic if it is an element of $\cA_s(X,R)\widehat{\otimes}_R M$. We write $\cA_s(X,M)$ for the space of such functions, which (by \cite[\S4]{Ser62}) inherits $R$-Banach module structure from the completed tensor product. Again, the inclusion maps $\cA_{s}(X,M) \subset \cA_{s+1}(X,M)$ are compact, and we let $\cA(X,M) = \varinjlim_s \cA_s(X,M)$. 

Recall the definition of \emph{orthonormalisable (ONable)} from \cite[\S A1]{Col97}. Any Banach space over $L$ is ONable \cite[Lem.\ 2.1.5]{Urb11}, so the spaces $\cA_s(X,L)$ and (when $M$ is a finite-dimensional $L$-vector space) $\cA_s(X,M)$ are ONable. When $R$ is a contractive $L$-algebra (for example, if $R$ is an $L$-affinoid algebra), then the completed tensor product of an ONable $L$-Banach space with $R$ is an ONable $R$-Banach module \cite[Prop.\ A1.3]{Col97}, so for such $R$ the spaces $\cA_s(X,R)$ are ONable. If $A$ and $B$ are two ONable Banach $R$-modules with ON bases $\{e_i\},\{f_j\}$, then $A\widehat{\otimes}_RB$ is an ONable Banach $R$-module with ON basis $\{e_i\otimes f_j\}$; hence when $R$ is contractive and $M$ is a finite Banach $R$-module, the spaces $\cA_s(X,M)$ are ONable.

For a Banach $R$-module $A$, let $\mathrm{Hom}_{R}(A,R)$ denote the space of continuous $R$-module maps $A \to R$. This is a Banach space via \cite[\S A1]{Col97}. If $R = L$, and $M$ is a finite-dimensional (normed) $L$-vector space, then we write $\cD_s(X,M) \defeq \mathrm{Hom}_L(\cA_s(X,M), L)$. In this case the maps $\cD_{s+1}(X,M) \subset \cD_{s}(X,M)$ are compact via the analogous statements for $\cA_s(X,M)$ and \cite[Lem.\ 16.4]{Sch02}, and (as it is an $L$-Banach space) $\cD_{s}(X,M)$ is ONable. Similar statements for dual spaces over $R$ are more subtle: see \S\ref{sec:distributions in families} below. 

\subsubsection{Analytic induction modules}
Let $Q = L_QN_Q$ be a parabolic. We may identify $J_Q$ with an open compact subset of $\Zp^r$ for some $r$, and thus apply the above formalism of analytic functions on $J_Q$. Let $M$ be a finite Banach $R$-module with a left-action of $L_Q(\Zp)$. We extend this action to $Q^-(\Zp) \cap J_Q = (J_Q^-L_Q)(\Zp)$ by dictating that $J_Q^-$ acts trivially. 
\begin{definition}
	Define the \emph{$s$-analytic induction} of $M$ to $J_Q$, denoted $\mathrm{LA}_s\mathrm{Ind}_Q M$, to be the space of functions $f : J_Q \to M$ such that $f \in \cA_s(J_Q,M)$ and
		\begin{equation}\label{eq:parahoric transform}
			f(bg) = b \cdot f(g)\text{ \ \ \  for all }b \in Q^-(\Zp)\cap J_Q\text{ and }g \in J_Q.
		\end{equation}
	We write $\mathrm{LAInd}_Q M$ for the space of such functions $f$ such that $f \in \cA(J_Q,M)$.
\end{definition}
 
Note that any such function $f$ is uniquely determined by its restriction to $N_Q(\zp)$ by \eqref{eq:parahoric transform} and the parahoric decomposition \eqref{eq:parahoric decomp}. Recall from \eqref{eq:RQ+} we have an explicit realisation of $N_Q(\Zp)$ as an open compact subset of $\Zp^t$ via the product decomposition $N_Q(\zp) \cong \prod_{\beta \in R^+ \backslash R_Q^+} \Zp X_\beta$. Note then that a function on $N_Q(\zp)$ is $s$-analytic if and only if it is analytic on each $N^s_Q(\Zp)$-coset.

\subsubsection{Locally analytic induction at single weights}
We recap the usual locally analytic modules. Here we take $Q$ to be the Borel $B$, with Levi $T$. Let $\lambda \in \cW(L)$ be a classical weight.

\begin{definition}
\begin{itemize}\setlength{\itemsep}{0pt}
		\item Denote the $s$-analytic induction of $\lambda$ by
		\[
			\cA_{\lambda,s}^G(L) \defeq \mathrm{LA}_s\mathrm{Ind}_B \lambda,
		\]
		realised as functions $f: J_B \to L$ with $f(n^-tg) = \lambda(t)f(g)$ for $t \in T(\Zp)$, $n^- \in N^-(\Zp)$.

	\item Let $\cA_{\lambda}^G(L) \defeq \mathrm{LAInd}_B \lambda = \varinjlim_s \cA_{\lambda,s}^G(L)$ be the module of locally analytic functions. 
	
\item We write $\cD_{\lambda,s}^G(L)$ and $\cD_{\lambda}^G(L)$ for the respective topological $L$-duals of the above spaces.
\end{itemize}
\end{definition}

The module $\cA_{\lambda,s}^G(L)$ can be identified with $\cA_s(N(\Zp),L)$ by restriction from $J_B$ to $N(\Zp)$, and inherits an $L$-Banach space structure from this space. Similarly the natural inclusions $\cA_{\lambda,s}^G(L) \subset \cA_{\lambda,s+1}^G(L)$ are all compact. Note also that via the restriction to $J_B$ explained in \S\ref{sec:algebraic induction}, we may view $V_\lambda^G(L)$ as the subspace of algebraic functions in $\cA_{\lambda,0}^G(L)$.

Now we work with a general $Q$, with Levi $L_Q$. Let $\lambda$ be a classical weight; it is also a weight for $L_Q$, and we have an algebraic $L_Q$-representation $V_{\lambda}^{L_Q}(L)$ of highest weight $\lambda$ via \S\ref{sec:algebraic induction}. 

\begin{definition}\label{def:parahoric distributions}
	\begin{itemize}\setlength{\itemsep}{0pt}
		\item Let $\A_{\lambda,s}^Q(L) \defeq \mathrm{LA}_s\mathrm{Ind}_{Q}[ V_{\lambda}^{L_Q}(L)]$.

\item Let $\A_{\lambda}^Q(L) \defeq \mathrm{LAInd}_Q [V_{\lambda}^{L_Q}(L)] = \varinjlim_s \A_{\lambda,s}^Q(L)$. 

\item 	Let $\D_{\lambda,s}^Q(L) \defeq \mathrm{Hom}_L(\A_{\lambda,s}^Q(L),L)$ and $\D_{\lambda}^Q(L) = \mathrm{Hom}_L(\A_{\lambda}^Q(L),L)$.
\end{itemize}
\end{definition}

As above, all the spaces with subscript $s$'s are Banach spaces over $L$. The spaces $\A_{\lambda}^Q(L)$ and $\D_{\lambda}^Q(L)$ are Fr\'{e}chet spaces, and $\D_{\lambda}^Q(L)$ is compact Fr\'echet in the sense of \cite[\S2.3.12]{Urb11}.
\begin{remark}
	As above, for any $Q$, the subspace of algebraic functions in $\A_{\lambda,0}^Q(L)$ is $V_{\lambda}^G(L)$, and hence $V_{\lambda}^{G,\vee}(L)$ is a quotient of $\D_{\lambda,0}^Q(L)$. At the extreme end, where we take  $Q=G$, then from the definition $\A_{\lambda,0}^G(L) = V_{\lambda}^G(L)$ and $\D_{\lambda,0}^G(L) = V_{\lambda}^{G,\vee}(L)$.
\end{remark}

\subsubsection{Integral structures}\label{sec:integral}
All of the above Banach spaces have natural integral structures, where we replace $L$ with $\cO_L$; in particular, as in \cite[3.2.13]{Urb11} we define 
\[
	\cA_{\lambda,s}^G(\cO_L) \defeq \cA_{\lambda,s}^G(L)\  \cap \ \cA_s(J_B, \cO_L),\hspace{20pt}
 	\A_{\lambda,s}^Q(\cO_L) \defeq \A_{\lambda,s}^Q(L) \ \cap \ \mathcal{A}_s\big(J_Q,V_{\lambda}^{L_Q}(\cO_L)\big).
\]
 The dual modules $V_{\lambda}^{G,\vee}(\cO_L)$, $\cD_{\lambda,s}^G(\cO_L)$ and $\D_{\lambda,s}^{Q}(\cO_L)$ are then all defined via $\cO_L$-duals.

\subsubsection{Analytic functions in families}

We now vary these spaces in families. Fix a classical weight $\lambda_0 \in \cW(L)$, and let $\cU \subset \cW_{\lambda_0}^Q$ be an affinoid (which we \emph{always} take to be admissible in the sense of \cite[Def.\ 2.2.6]{Con08}, so that it is open in the Tate topology on $\cW_{\lambda_0}^Q$). If $\lambda \in \cU(L)$, then by definition $\lambda\lambda_0^{-1} \in \cW^{Q}(L)$ is a character of $L_Q(\Zp)$.

\begin{lemma}\label{lem:changing weight}If $\lambda \in \cU(L)$ is classical, then we have an isomorphism of $L_Q(\Zp)$-modules 
\[
	V_{\lambda}^{L_Q}(L) \cong V_{\lambda_0}^{L_Q}(L) \otimes_L \lambda\lambda_0^{-1}.
\]
\end{lemma}
\begin{proof}
	The character $\lambda\lambda_0^{-1}$ can itself, as an irreducible representation of $L_Q$, be viewed as the highest weight representation $V_{\lambda\lambda_0^{-1}}^{L_Q}$. Then $V_\lambda^{L_Q} = V_{\lambda_0\lambda\lambda_0^{-1}}^{L_Q}$ is a subrepresentation of the tensor product; but the tensor product of an irreducible representation with a character is irreducible.
\end{proof} 

Crucial for variation is the fact that the underlying spaces of $V^{L_Q}_{\lambda}(L)$ and $V^{L_Q}_{\lambda_0}(L)$ are the same: only the $L_Q(\Zp)$-action is different. We now vary the action analytically.

As $\cW$ is a rigid analytic group, translation by $\lambda_0$ defines a rigid analytic automorphism of $\cW$. Let $\cU_0 \defeq \lambda_0^{-1}\cU \subset \cW^Q$; this translation identifies $\cU_0$ isomorphically with $\cU$, so it is an affinoid defined over $L$. 
Attached to such an affinoid, there exists a tautological/universal character $\chi_{\cU_0} : L_Q(\Zp) \longrightarrow \cO(\cU_0)^\times$ with the property that for each weight $\lambda\lambda_0^{-1} \in \cU_0(L)$, composing $\chi_{\cU_0}$ with evaluation $\cO(\cU_0) \to L$ at $\lambda\lambda_0^{-1}$ recovers the corresponding map $L_Q(\Zp) \rightarrow L^\times$. Necessarily such a character must factor through the abelianisation $L_Q^{\mathrm{ab}}(\Zp)$, and $L_Q^{\mathrm{ab}}$ (as a commutative reductive group) is a torus. Any character of $L_Q^{\mathrm{ab}}(\Zp)$ is then locally analytic by \cite[Prop.\ 8.3]{Buz07}. We deduce $\chi_{\cU_0}$ is the composition of a locally analytic map with the analytic (even algebraic) map $L_Q \to L_Q^{\mathrm{ab}}$, hence it is $s$-analytic for all $s$ greater than some (minimal) integer $s[\cU]$.

\begin{definition}
	Define a finite free $\cO(\cU_0)$-module $V_{\cU}^{L_Q} \defeq V^{L_Q}_{\lambda_0}(L) \otimes_L \cO(\cU_0)$, and a map
\begin{align*}
	\langle \cdot \rangle_{\cU} : L_Q(\Zp) &\longrightarrow \mathrm{Aut}\left(V_{\lambda_0}^{L_Q}(L)\right) \otimes_L \cO(\cU_0)^\times \subset \mathrm{Aut}\left(V^{L_Q}_{\cU}\right)\\
		\ell & \longmapsto \langle \ell \rangle_{\lambda_0} \otimes \chi_{\cU_0}(\ell).
\end{align*}
\end{definition}

This makes $V^{L_Q}_{\cU}$ into an $L_Q(\Zp)$-representation. From the definition of $\chi_{\cU_0}$, we deduce: 
\begin{proposition}
	For any classical $\lambda \in \cU(L)$, evaluation $\cO(\cU_0) \rightarrow L$ at $\lambda\lambda_0^{-1} \in \cU_0$ induces a surjective map
\[
	\mathrm{sp}_{\lambda} : V^{L_Q}_{\cU} \longrightarrow V^{L_Q}_{\lambda_0}(L) \otimes_L \lambda\lambda_0^{-1} \cong V^{L_Q}_{\lambda}(L)
\]
of $L_Q(\Zp)$-representations. Thus $V^{L_Q}_{\cU}$ interpolates the representations $V^{L_Q}_{\lambda}(L)$ as $\lambda$ varies in $\cU$. 
\end{proposition}
Here $\mathrm{sp}_\lambda$ is surjective since $\cO(\cU_0) \twoheadrightarrow L$ (evaluation at $\lambda\lambda_0^{-1}$) is surjective and $V_{\lambda_0}^{L_Q}(L)$ is $L$-flat.

\begin{remark}
	The choice of $\lambda_0$ fixes an identification of $\cU$ and $\cU_0$, and hence of $\cO(\cU)$ and $\cO(\cU_0)$, which is compatible with our normalisation of specialisation maps. Henceforth we work only with $\cU$, and implicitly the transfer of structure is with respect to this choice of identification.
\end{remark}

\begin{definition}
	For any $s \geq s[\cU]$, define\footnote{Note this is only well-defined for $s \geq s[\cU]$ since otherwise the action of $L_Q(\Zp)$ is not $s$-analytic.} 
	\[
		\A_{\cU,s}^Q \defeq \mathrm{LA}_s\mathrm{Ind}_{Q} V^{L_Q}_{\cU}, \hspace{10pt}\text{and}\hspace{10pt} \A_{\cU}^Q = \varinjlim_s \A_{\cU,s}^Q.
	\]
\end{definition}

\begin{lemma}\label{lem:AU ON}
	 $\A_{\cU,s}^Q \cong \cA_s(N_Q(\Zp),V_{\lambda_0}^{L_Q}(L)) \widehat{\otimes}_L\cO(\cU)$ is an ONable $\cO(\cU)$-Banach module. If $\{e_i\}_{i\in\N}$ is an ON basis of $\A_{s}(N_Q(\Zp),V_{\lambda_0}^{L_Q}(L))$, then $\{e_i \otimes 1\}_{i\in\N}$ is an ON basis of $\A_{\cU,s}^Q$.
\end{lemma}
\begin{proof}
	By \eqref{eq:parahoric transform}, restriction to $N_Q(\Zp)$ gives $\A_{\cU,s}^Q \cong \cA_s(N_Q(\Zp),V_{\cU}^{L_Q})$, which from the definitions is isomorphic to $\cA_s(N_Q(\Zp), V_{\lambda_0}^{L_Q}(L)) \widehat{\otimes}_L \cO(\cU)$. The rest now follows from \cite[Prop.\ A1.3]{Col97}.
\end{proof}

\subsubsection{Distributions in families}\label{sec:distributions in families}
 
 Since $\cO(\cU)$-duals are not as well-behaved as $L$-duals, we have to work harder to study the distributions in this setting. See e.g.\ \cite[Rem.\ 3.1]{Bel12} or \cite[\S2.2]{Han17} for analogous discussions.  The first natural space to study is the Banach/continuous dual
 \[ 
 \D_{\cU,s}^Q \defeq \mathrm{Hom}_{\cO(\cU)}\big(\A_{\cU,s}^Q, \cO(\cU)\big).
 \]
The natural restriction maps $\D_{\cU,s+1}^Q \to \D_{\cU,s}^Q$ are injective (as in \cite[\S2.2]{Han17}). However this is not obviously ONable. Since we require this for slope decompositions, we also define
 \[
 \widetilde{\D}_{\cU,s}^Q \defeq \cD_s\big(N_Q(\Zp),V_{\lambda_0}^{L_Q}(L)\big) \hspace{2pt}\widehat{\otimes}_{L}\cO(\cU).
 \]
Since $\cO(\cU)$ is a contractive Banach $L$-algebra, this space is an ONable Banach $R$-module \cite[Prop.\ A1.3]{Col97}; and the restriction maps $\widetilde{\D}_{\cU,s+1}^Q \rightarrow \widetilde{\D}_{\cU,s}^Q$ are compact by \cite[Cor.\ 2.9]{Buz07}.  By formalism of duals/tensor products there is a natural inclusion
\begin{equation}\label{eq:inclusion of dists}
	r_s : \widetilde{\D}_{\cU,s}^Q \hookrightarrow \D_{\cU,s}^Q,
\end{equation}
defined on pure tensors by $\mu \otimes \alpha \mapsto [(f\otimes \beta) \mapsto \mu(f)\alpha\beta]$ (using Lemma \ref{lem:AU ON}). Then (cf.\ \cite[\S2.2]{Han17}):
	\begin{lemma}\label{lem:j_s}
		For all $s$, there is a compact injective map $j_s : \D_{\cU,s+1}^Q \hookrightarrow \widetilde{\D}_{\cU,s}^Q$ making the following diagram commute:
		\begin{equation}\label{eq:j_s}
		\xymatrix@C=15mm{
			\widetilde{\D}_{\cU,s+1}^Q \ar[r]^{\mathrm{res}}\ar[d]^{r_s}& \widetilde{\D}_{\cU,s}^Q\ar[d]^{r_s}\\
			\D_{\cU,s+1}^Q \ar[r]^{\mathrm{res}}\ar[ru]^{j_s}& \D_{\cU,s}^Q
		}
		\end{equation}
	\end{lemma}	
	\begin{proof}
		If the map exists, it is compact (by considering the top triangle) and injective (by considering the bottom triangle). To prove existence, let $\{e_i\}_{i \in \N}$ be an ON basis of $\cA_s(N_Q(\Zp),V_{\lambda_0}^{L_Q}(L))$. For $i \in \N$ define distributions
		\[
			\nu_i \in \cD_s\big(N_Q(\Zp),V_{\lambda_0}^{L_Q}(L)\big), \hspace{18pt} \nu_i(e_j) = \left\{\begin{smallmatrix}1 &: i=j\\
				0 &: i\neq j,\end{smallmatrix}\right.
		\]
extended continuously. Then given $\mu \in \D_{\cU,s+1}^Q$, we define
\begin{equation}\label{eq:sum j_s}
	j_s(\mu) \defeq \sum_{i \in \N} \nu_i \otimes \mu(e_i\otimes 1) \in \widetilde{\D}_{\cU,s}^Q.
\end{equation}
To see this is well-defined, note $\{e_i \otimes 1\}$ is an ON basis of $\A_{\cU,s}^Q$ by Lemma \ref{lem:AU ON}. As the inclusion $\A_{\cU,s}^Q \subset \A_{\cU,s+1}^Q$ is compact, the sequence $e_i \otimes 1$ tends to zero in $\A_{\cU,s+1}^Q$, and $\mu(e_i\otimes 1) \to 0$; hence the sum in \eqref{eq:sum j_s} converges in the completed tensor product. Commutativity of \eqref{eq:j_s} follows easily from the definitions.
	\end{proof}

\begin{definition}\label{def:compact frechet}
	Define the space of \emph{parahoric locally analytic distributions over $\cU$} to be
	\[
		\D_{\cU}^Q \defeq \mathrm{Hom}_{\cO(\cU)}\big(\A_{\cU}^Q, \cO(\cU)\big).
	\]
\end{definition}

\begin{lemma}\label{lem:compact frechet}
		The space $\D_{\cU}^Q \cong \varprojlim_s \D_{\cU,s}^Q \cong \varprojlim_s \widetilde{\D}_{\cU,s}^Q$ is a compact Fr\'{e}chet $\cO(\cU)$-module.
\end{lemma}

\begin{proof}
The first isomorphism is standard, and the second isomorphism (between inverse limits) follows from Lemma \ref{lem:j_s}. We conclude since $\varprojlim_s \widetilde{\D}_{\cU,s}^Q$ is compact Fr\'echet by definition.
\end{proof}

\begin{remark}\label{rem:non-classical}
	If $\lambda\in \Wlam$ is any (possibly non-classical) weight, then we may still define an $L_Q(\Z_p)$-module $V_\lambda^{L_Q}(L) \defeq V_{\lambda_0}^{L_Q}(L) \otimes \lambda\lambda_0^{-1}$. Hence we can define $\D_{\lambda,s}^Q(L)$ and $\D_{\lambda}^Q(L)$ identically to Definition \ref{def:parahoric distributions}. Note $V_{\lambda}^{L_Q}(L)$ is independent of the choice of base weight $\lambda_0$, since if $\lambda_0'$ is another choice, by Lemma \ref{lem:changing weight} (in the first isomorphism) we have
		\[
		V_{\lambda}^{L_Q}(L) \defeq V_{\lambda_0}^{L_Q}(L) \otimes_L\lambda\lambda_0^{-1} \cong [V_{\lambda_0'}^{L_Q}(L) \otimes_L \lambda_0(\lambda_0')^{-1}] \otimes_L\lambda\lambda_0^{-1} \cong V_{\lambda_0'}^{L_Q}(L)\otimes_{L}\lambda(\lambda_0')^{-1}.
		\]	
		Hence $\D_{\lambda,s}^Q(L)$ and $\D_{\lambda}^Q(L)$ are also independent of the choice of $\lambda_0$. 
\end{remark}

\begin{remark}
If $\cU' \subset \cU$ is a closed affinoid subspace, then (by definition of $\widetilde{\D}_{\cU,s}^Q$) we have $\D_{\cU}^Q \otimes_{\cO(\cU)}\cO(\cU') \cong \D_{\cU'}^Q$. If $\lambda \in \cU(L)$ corresponds to the maximal ideal $\m_\lambda \subset \cO(\cU)$, we thus have 
\[
	\D_{\cU}^Q \otimes_{\cO(\cU)} \cO(\cU)/\m_\lambda \cong\D_{\lambda}^Q(L),
\]
and a specialisation map $\mathrm{sp}_\lambda : \D_{\cU}^Q \twoheadrightarrow \D_{\lambda}^Q(L)$. Thus $\D_{\cU}^Q$ interpolates $\D_{\lambda}^Q(L)$ as $\lambda$ varies in $\cU(L)$.
\end{remark}

\subsection{Summary of notation}\label{sec:summary of notation}
The notation in the above is heavy. To ease notation, henceforth we will fix a coefficient field $L/\Qp$, containing the fixed splitting field $E$ of $G$, and drop it from the notation, writing $\A^Q_{\lambda,s} = \A^Q_{\lambda,s}(L)$, $V_\lambda^G = V_{\lambda}^G(L)$, etc. 

In Table \ref{table of coeffs} we give a brief key of our notation in the language of \S\ref{sec:analytic functions}.  Note that \emph{all} of the analytic function spaces can be characteristed uniquely by their restrictions to a unipotent subgroup, valued in some Banach module, and then extended uniquely to $J_B$ or $J_Q$ using the weight action. For a classical weight $\lambda$ and any $s \geq 0$, we get the chain of modules
\begin{equation}
	\xymatrix@C=4mm@R=3mm{
			(\text{Banach}) &&V_{\lambda}^G \sar{r}{=} & \A_{\lambda,0}^G \sar{r}{\subset}&  \A_{\lambda,s}^Q \sar{r}{\subset}\sar{d}{\subset}& \A_{\lambda,s}^B \sar{r}{=}\sar{d}{\subset}& \cA_{\lambda,s}^G\sar{d}{\subset}\\
		(\text{Fr\'echet})&&&& \A^Q_{\lambda}\sar{r}{\subset}& \A_{\lambda}^B \sar{r}{=} & \cA_{\lambda}^G.
		}
\end{equation}
The notation we maintain is that $\mathbf{A}^Q$ means $Q$-parabolic induction and $\mathcal{A}^G$ means full induction. Modules with subscripts $s$ are Banach modules, and $s$ denotes the degree of analyticity; those without a subscript $s$ are Fr\'echet modules. Despite the equality $\A_{\lambda,s}^B = \cA_{\lambda,s}^G$, we choose to maintain the separate notation $\A$ and $\cA$ both for clarity and because the modules $\cA^{L_Q}_{\lambda,s}$ play a crucial role in the sequel.

\renewcommand{\arraystretch}{1.4}
\begin{table}[h!]
	\centering
	\begin{tabular}{c|c|c|c|c}
		\textbf{Module}  & \textbf{On unipotent } & \textbf{Extension} & \textbf{Dual}&\textbf{Nomenclature} \\ \hline
		$V_{\lambda}^G$	&  - & - & $V_{\lambda}^{G,\vee}$ & algebraic on $G$\\ \hline
		$\cA_{\lambda,s}^G$	&  $\cA_s(N(\Zp),L)$ & $f : J_B \to L$ & $\cD_{\lambda,s}^G$ & $s$-analytic on $N$\\ \hline
		$\cA_{\lambda}^G$	&  $\cA(N(\Zp),L)$ & $f : J_B \to L$ & $\cD_{\lambda}^G$ & locally analytic on $N$\\ \hline
	$\A_{\lambda,s}^Q$ & $\cA_s\left(N_Q(\Zp),V_{\lambda}^{L_Q}\right)$ & $f : J_Q \to V_{\lambda}^{L_Q}$ & $\D_{\lambda,s}^Q$ & $s$-an.\ on $N_Q$, $s$-alg.\ on $L_Q$ \\ \hline
		$\A_{\lambda}^Q$	&  $\cA\left(N_Q(\Zp),V_{\lambda}^{L_Q}\right)$ & $f : J_Q \to V_{\lambda}^{L_Q}$ & $\D_{\lambda}^Q$& loc.\ \!an.\ on $N_Q$, loc.\ \!alg.\ on $L_Q$ \\ \hline
		$\A_{\cU}^Q$	&  $\cA\left(N_Q(\Zp),V_{\cU}^{L_Q}\right)$ & $f : J_Q \to V_{\cU}^{L_Q}$ & $\D_{\cU}^Q$& loc.\ \!an.\ on $N_Q$, loc.\ \!alg.\ on $L_Q$ 
	\end{tabular}
\caption{\label{table of coeffs} \emph{Modules of coefficients.} }
\end{table}

%%==============================================
%%==============================================
\subsection{The action of $\Sigma_Q$ and local systems}\label{sec:Sigma action}
\begin{definition}
	Let $\Sigma_Q$ denote the monoid in $G(\Qp)$ generated by $J_Q$ and $T^+$. 
\end{definition}

Let $\diamond$ denote either a single classical weight $\lambda$ or an affinoid $\cU$ in $\Wlam$ for a fixed classical $\lambda_0$. The parahoric $J_Q$ acts on itself by right multiplication, which then give rise to left actions of $J_Q$ on $\A_{\diamond,s}^Q$ and $\A_{\diamond}^Q$ and dual right actions on $\D_{\diamond,s}^Q$ and $\D_{\diamond}^Q$.

The action of $T^+$ is more subtle; we note that any function $f \in \A_{\diamond,s}^Q$ is uniquely determined by its restriction to $B(\Zp)$, upon which $t \in T^+$ acts by $b \mapsto t^{-1}bt$ (by Proposition \ref{prop:valuation T++}(iii)). In itself, this is not compatible with the action of $J_Q$ above due to the left multiplication by $t^{-1}$. To rectify this, note that our choice of uniformisers defines a splitting 
\begin{equation}\label{eq:splitting}
T(\Qp) \isorightarrow T(\Zp) \times T(\Qp)/T(\Zp), \hspace{12pt} t \mapsto (\sigma(t),\zeta(t)).
\end{equation} 
Also write $\zeta$ for the composition $T(\Qp) \xrightarrow{\zeta} T(\Qp)/T(\Zp) \hookrightarrow T(\Qp)$. Then $T^+$ acts on $B(\Zp)$ by
\begin{equation}\label{eq:action of t}
	b * t =  \zeta(t)^{-1}bt = \sigma(t)t^{-1}bt.
\end{equation}
Now if $t \in T(\Zp) = T(\Qp) \cap J_Q$, then $\sigma(t)t^{-1} = 1$ and \eqref{eq:action of t} coincides with right translation by $t$. If $f \in \A_{\diamond,s}^Q$, define $t*f$ on $B(\Zp)$ by $(t*f)(b) = f(b*t)$, and extend to $J_Q$ via \eqref{eq:parahoric transform}. A simple check shows $t*f \in \A_{\diamond,s}^Q$ is well-defined, giving a left action of $T^+$ on $\A_{\diamond,s}^Q$ and a right action on $\D_{\diamond,s}^Q$. 

\begin{notation}
 	If $g \in \Sigma_Q$, denote the action of $g$ on $f \in \A_{\diamond,s}^Q$ by $g * f$, and on $\mu \in \D_{\diamond,s}^Q$ by $\mu * g$. 
\end{notation}

	\begin{lemma}
		 The image of the map $r_s : \widetilde{\D}_{\cU,s}^Q \hookrightarrow \D_{\cU,s}^Q$ from \eqref{eq:inclusion of dists} is preserved by $\Sigma_Q$. 
	\end{lemma}
	\begin{proof}
		We can argue exactly as in \cite[Rem.\ 3.1]{Bel12}. Alternatively, we can directly write down an action on $\widetilde{\D}_{\cU,s}^Q$: let $j \in J_Q, t \in T^+$ and $\mu\otimes\alpha \in \D_{\lambda_0,s}^Q(L)\widehat{\otimes}_L \cO(\cU_0)$, which we identify with $\widetilde{\D}_{\cU,s}^Q$ via restriction to $N_Q(\Zp)$. Write $j = j^-\ell_j n_j$ under \eqref{eq:parahoric decomp}. On pure tensors, define
		\[
			(\mu \otimes \alpha)*j = (\mu*j) \otimes \chi_{\cU_0}(\ell_j)\alpha, \hspace{12pt} (\mu\otimes \alpha)*t = (\mu*t) \otimes \chi_{\cU_0}(\sigma(t))\alpha,
		\]
		extended by continuity. One may check explicitly that \eqref{eq:inclusion of dists} is equivariant for the $*$-actions. 
\end{proof}

Suppose $K \subset \cG(\A_f)$ is open compact with $K_p \subset J_Q$. Via projection to $K_p$, these spaces of locally analytic distributions are $K$-modules which then, via \eqref{eq:local systems}, give local systems over the locally symmetric space, which in a slight abuse of notation we denote by the same symbols.

\begin{definition} 
	The \emph{parahoric overconvergent cohomology groups} (with respect to the parabolic $Q$) are the groups  $\h^i_{\mathrm{c}}(S_{K},\D_{\diamond,s}^Q)$, $\h^i_{\mathrm{c}}(S_{K},\D_{\diamond}^Q)$ and  $\hc{i}(S_K,\widetilde\D_{\cU,s}^Q)$.
\end{definition}

 The action of $t \in T^+$ then allows us to define Hecke operators $U_t$ on the parahoric overconvergent cohomology groups, exactly as in \cite[\S2.1]{Han17}. We extend this to an action of $\uhp(K)$ by letting $\cG'(F_v)$ act trivially on $\D_{\diamond,s}^Q$ for all $v\nmid p$.
 
 \begin{remarks}\label{rem:dot action}
 	\begin{itemize}\setlength{\itemsep}{0pt}
 \item[(i)]	Note that more or less by definition, the $*$-action of $\Sigma_Q$ defined here preserves the integral subspaces $\D_{\lambda,s}^Q(\cO_L)$ of \S\ref{sec:integral}.
 \item[(ii)] The $*$-action also preserves algebraic subspaces. In particular, we get a $*$-action of $\Sigma_Q$ on $V_{\lambda}^G(L)$ which preserves $V_{\lambda}^G(\cO_L)$. But any $f \in V_{\lambda}^G(L)$ extends uniquely from $G(\Zp)$ to $G(\Qp)$, from which we get a natural `algebraic' action of $G(\Qp)$ defined by $(t \cdot f)(g) \defeq f(gt)$. From the definition, we find that for $f \in V_{\lambda}^G$ and $t \in T^+$, we have
 \begin{equation}\label{eq:modify}
 	(t * f)(g) = f(\sigma(t)t^{-1}gt) = \lambda\big(\sigma(t)t^{-1}\big)(t \cdot f)(g)
 \end{equation}
(compare \cite[(15)]{Urb11}). The $\cdot$-action does \emph{not} preserve $V_{\lambda}^G(\cO_L)$, and the $*$-action can be viewed as an `optimal' integral normalisation of it.

\item[(iii)] For $\GL_2$, it is easy to write down the $\cdot$-action on $V_{\lambda}^{G,\vee}$ explicitly, and one easily sees that this explicit action extends to distributions; this is done, for example, in \cite{PS11,Bel12,BW_CJM}. We warn the reader, however, that this does not give the $*$-action of $T^+$ on distributions defined here: in particular, it does not preserve integrality (see \cite[\S9.1]{BW_CJM}).
 \end{itemize}
For the remainder of this paper, unless explicitly stated, all actions will be the $*$-actions.
 \end{remarks}

%%==============================================

\subsection{Compact operators and slope decompositions} 
We now recap the (standard) arguments that show the parahoric overconvergent cohomology groups admit slope decompositions with respect to $Q$-controlling operators.

\begin{lemma}\label{lem:compact}
	If $t \in T^{++}_Q$, then $t$ acts compactly on $\D_{\lambda,s}^Q$ and $\D_{\lambda}^Q$, and on $\widetilde{\D}_{\cU,s}^Q$ and $\D_{\cU}^Q$. 
\end{lemma}
\begin{proof} At a single weight $\lambda$, we follow \cite[Lemma 3.2.8]{Urb11}. Firstly, since by definition of $T_Q^{++}$ we have $t^{-1}N_Q^s(\Zp)t \subset N_Q^{s+1}(\Zp)$, we see that $t \cdot \A^Q_{\lambda,s+1}(L) \subset \A_{\lambda,s}^Q(L)$. Hence on distributions, we have $\D_{\lambda,s}^Q(L) \cdot t \subset \D_{\lambda,s+1}^Q(L)$ (that is, $t$ improves the analyticity). Thus the action of $t$ factors through the (compact) inclusion map $\D_{\lambda,s+1}^Q(L) \hookrightarrow \D_{\lambda,s}^Q(L)$. This ensures that it acts compactly on $\D_{\lambda,s}^Q(L)$, and also the limit $\D_{\lambda}^Q(L)$ by definition. The statements for $\cU$ then follow combining this with \cite[Lem.\ 2.9]{Buz07}, the definition of $\widetilde{\D}_{\cU,s}^Q$ (from \S\ref{sec:distributions in families}), and Lemma \ref{lem:compact frechet}.
\end{proof}

If $M$ is a module admitting a slope $\leq h$ decomposition with respect to an operator $U$ (see, for example, \cite[Definition 2.3.1]{Han17}), we write it as
\begin{equation}\label{eq:slope decomp}
M = M^{U\leq h} \oplus M^{U > h}.
\end{equation}

Let $\hc{\bullet}$ denote compactly supported (Betti) cohomology, dual to the Borel--Moore homology. The following adaptation of \cite[\S4]{AS08} is the main reason we introduced the (ONable) spaces $\widetilde{\D}_{\cU,s}^Q$.

\begin{proposition} \label{prop:slopes} 
	Let $K$ be an open compact subgroup of $\cG(\A_f)$ with $K_p \subset J_Q$, let $\cU \subset \Wlam$ be an open affinoid, let $h \geq 0$, and let $t \in T^{++}_Q$.  Then, possibly up to replacing $\cU$ with a smaller affinoid neighbourhood of $\lambda$:
	\begin{itemize}\setlength{\itemsep}{0pt}
		\item[(i)] The spaces $\h_{\mathrm{c}}^\bullet(S_K,\D_{\lambda,s}^Q)$ and $\hc{\bullet}(S_K,\widetilde{\D}_{\cU,s}^Q)$ admit slope $U_t \leq h$ decompositions for all $s$. 
		\item[(ii)] The small slope parts $\hc{\bullet}(S_K,\D_{\lambda,s}^Q)^{U_t \leq h}$ and $\hc{\bullet}(S_K,\widetilde{\D}_{\cU,s}^Q)^{U_t \leq h}$ are independent of $s$.
		\item[(iii)] Both $\hc{\bullet}(S_K,\D_{\lambda}^Q)$ and $\hc{\bullet}(S_K,\D_{\cU}^Q)$ admit slope $U_t \leq h$ decompositions, and for any $s$
		\[
			\hc{\bullet}(S_K,\D_{\lambda}^Q)^{U_t \leq h} \cong \hc{\bullet}(S_K,\D_{\lambda,s}^Q)^{U_t\leq h}, \hspace{12pt} \hc{\bullet}(S_K,\D_{\cU}^Q)^{U_t \leq h} \cong \hc{\bullet}(S_K,\widetilde{\D}_{\cU,s}^Q)^{U_t\leq h}.
		\]
	\end{itemize}
\end{proposition}

\begin{proof} 
	These results are all standard, so we only give analogous references. The modules we have defined give rise to compactly supported chain complexes $C_{\mathrm{c}}^\bullet(K,\D_{\lambda,s}^Q)$ and $C_{\mathrm{c}}^\bullet(K,\widetilde{\D}_{\cU,s}^Q)$, as at the end of \cite[\S3]{Han17}, and the compactness of $t$ on distributions lifts to compactness of $t$ on the complex. The cohomology of this complex gives rise to the compactly supported cohomology groups in which we are primarily interested. Since the $\D_{\lambda,s}^Q$ and $\widetilde{\D}_{\cU,s}^Q$ are ONable, Propositions 2.3.3--2.3.5 of \cite{Han17} then show part (i). Part (ii) is the parahoric analogue of Proposition 3.1.5 \emph{op.\ cit}., arguing identically using instead the parahoric chain complexes. Part (iii) follows in the inverse limit (using Lemma \ref{lem:compact frechet} for distributions over $\cU$).
\end{proof}

Note that, directly from the definitions, if $M$ is a $\Qp$-module that admits a slope decomposition with respect to an operator $U$, and $\beta \in \Qp$, then 
\begin{equation}\label{eq:scaled slope}
	M^{(\beta U)\leq h} \cong M^{U \leq [h - v_p(\beta)]}.
\end{equation}

%%============================================================
%%
%%			PARABOLIC L.A. BGG
%%
%%============================================================

\section{Parahoric classicality theorems}\label{sec:classicality}
We now prove our central result, a relative classicality theorem for parahoric overconvergent cohomology. This encompasses the analogous theorem for lifting from fully algebraic to fully analytic coefficients, and indeed we expect that it gives a numerically optimal slope bound for such a result. Our main tool is a parahoric version of Jones and Urban's locally analytic Bernstein--Gelfand--Gelfand (BGG) resolution for classical weights $\lambda$ (Corollary \ref{cor:loc an BGG 2}), which we develop in \S\ref{sec:usual BGG}-\ref{sec:la BGG}. This can also be considered as a locally analytic version of the main result of \cite{Lep77}.

As in \S\ref{sec:summary of notation}, we fix a coefficient field $L/\Qp$, containing $E$ splitting $G$, and drop it from notation.

\subsection{The parahoric classicality theorem}\label{sec:Q-classicality} 
Fix throughout this section a parabolic $Q \subset G$, an open compact $K \subset \cG(\A_f)$ with $K_p \subset J_Q$ and a classical weight $\lambda$. Dualising the natural inclusion $V_{\lambda}^G$ $\subset \A_{\lambda,0}^Q \subset \A_{\lambda,s}^Q$ yields a map $\D_{\lambda,s}^Q \to V_{\lambda}^{G,\vee},$ and a corresponding map on cohomology:
\begin{equation}
\rho_\lambda : 	\h^\bullet_{\mathrm{c}}\big(S_{K}, \D_{\lambda,s}^Q\big) \longrightarrow \h^\bullet_{\mathrm{c}}\big(S_{K},V_{\lambda}^{G,\vee}\big).
\end{equation}

\begin{definition} \label{def:Q non-critical} Let $\phi$ be a system of Hecke eigenvalues (for $\uhp(K)$) occurring in $\hc{\bullet}(S_K,V_{\lambda}^{G,\vee})$. We say $\phi$ is \emph{$Q$-non-critical} if the map $\rho$ restricts to an isomorphism of $\phi$-generalised eigenspaces
	\[
	\h^\bullet_{\mathrm{c}}\big(S_{K}, \D_{\lambda,s}^Q\big)_{\phi} \isorightarrow \h^\bullet_{\mathrm{c}}\big(S_{K},V_{\lambda}^{G,\vee}\big)_{\phi}.
	\]
\end{definition}

	Such systems $\phi$ naturally arise from `$p$-refined' automorphic representations $\tilde\pi$; see \S\ref{sec:warning}. We say such a $\tilde\pi$ is $Q$-non-critical if the associated $\phi$ is. We observe that for finite slope systems, this definition has no dependence on the radius of analyticity $s$, so is well-defined; and in fact we may pass to distributions that are fully locally analytic in $Q$:

\begin{lemma}
	Let $\phi$ be a $Q$-non-critical system of Hecke eigenvalues, and assume $\phi$ has $Q$-finite slope (i.e.\ $\phi(U_t) \neq 0$ for some $t \in T_Q^{++}$). Then for any $s \geq 0$, we have
	\[
	\hc{\bullet}\big(S_K,\D_{\lambda}^Q\big)_{\phi} \cong \hc{\bullet}\big(S_K,\D_{\lambda,s}^Q\big)_{\phi} \cong \h^\bullet_{\mathrm{c}}\big(S_{K},V_{\lambda}^{G,\vee}\big)_{\phi}
	\]
\end{lemma}
\begin{proof}
	This follows from Proposition \ref{prop:slopes} applied with some $h \geq v_p(\phi(U_t))$. 
\end{proof}

\begin{definition}
	For $\lambda$ a classical weight, $t \in T^+$ and $\alpha \in \Delta$, let
	\begin{align*}\label{eq:non-crit slope}
	h^{\mathrm{crit}}(t,\alpha,\lambda) &\defeq v_p\left(t^{w_\alpha*\lambda - \lambda}\right)\\ 
	&= -\big[\langle \lambda, \alpha^\vee\rangle + 1\big] \cdot v_p(\alpha(t)).
	\end{align*}
	Here $t^\lambda \defeq \lambda(t)$, and the equality is \eqref{eq:weyl action on weights}.
\end{definition}

This provides a numerical criterion for $Q$-non-criticality. Define a maximal chain of parabolics
\[
Q = P_0 \subset P_{1} \subset \cdots \subset P_m = G
\]
containing $Q$, so that $\Delta_{P_{i}} = \Delta_{P_{i-1}} \cup \{\alpha_{i}\}$ for some simple root $\alpha_{i}$. For each $i = 1, ..., m$, let $t_i \in T^+$ such that $v_p(\alpha_i(t_i)) < 0$, and let $U_i = U_{t_i}$. The rest of \S\ref{sec:classicality} will be dedicated to proving:

\begin{theorem}\label{thm:Q-classicality}
	Let $\phi$ be as in Definition \ref{def:Q non-critical}. Suppose $\phi$ is \emph{$Q$-non-critical slope} in the sense that
	\[
	h_i \defeq v_p\big[\phi(U_{i})\big] < h^{\mathrm{crit}}(t_i,\alpha_i,\lambda)
	\]
	for all $i = 1,...,m$. Then $\phi$ is $Q$-non-critical.
\end{theorem}

\begin{examples}\label{ex:classicality}
	\begin{itemize}\setlength{\itemsep}{0pt}
		\item Let $G = \GL_n$, with root system $A_{n-1}$ and simple roots $\Delta = \{\alpha_1,...,\alpha_{n-1}\}$. We get a chain of parabolics $P_i$ corresponding to $\varnothing \subset \{\alpha_1\} \subset \cdots \subset \{\alpha_1,...,\alpha_{n-1}\} = \Delta$ : precisely,  $P_{0} = B$ and if $i > 0$, then $P_i$ is the parabolic with Levi $\GL_{i+1} \times \GL_1^{n-i-1}$. We may take $t_i = \mathrm{diag}(1,...,1,p,...,p)$, with $p$'s in the last $n-i$ entries; then for $\lambda = (\lambda_1,...,\lambda_n)$, we have $h^{\mathrm{crit}}(t_i,\alpha_i,\lambda) = \lambda_i - \lambda_{i+1} + 1$. Thus a $p$-refined automorphic representation $\tilde\pi$ of $\GL_n(\A)$ with $U_i$-eigenvalues $A_i$ is $B$-non-critical if $v_p(A_i) < \lambda_i - \lambda_{i+1}+ 1$ for $i = 1,..,n-1$. If we just use the usual $U_p$-operator corresponding to $\mathrm{diag}(1,p,...,p^{n-2},p^{n-1})$, with eigenvalue $A$, then the small slope criterion is (the much more restrictive) $v_p(A) < \mathrm{min}_i(\lambda_i - \lambda_{i+1} + 1)$.
		
		\item In the same set-up, if $Q_j$ is the maximal standard parabolic with $\Delta_{Q_j} = \Delta \backslash \{\alpha_j\}$, then $\tilde\pi$ is $Q_j$-non-critical slope if $v_p(A_j) < h^{\mathrm{crit}}(t_j,\alpha_j,\lambda) = \lambda_j - \lambda_{j+1} + 1$.

		\item Let $G = \mathrm{GSp}_4$, with root system a generalised form of $C_2$ (with an additional basis vector $e_3$ for the character space; see \cite[\S2.3]{RS07} for more details). Let $\f$ be a Siegel eigenform with weight $\lambda = (k_1+3,k_2+3)$ with $k_1 \geq k_2 \geq 0$; then $\f$ is cohomological. This corresponds to the character $k_1 e_1 + k_2 e_2 + 0e_3$. The simple roots are $\alpha_1 = e_1 - e_2$ and $\alpha_2 = 2e_2- e_3$, with coroots $\alpha_1^\vee = e_1^\vee - e_2^\vee$ and $\alpha_2^\vee = e_2^\vee$. There are thus two non-minimal parabolics:
		\begin{itemize}
			\item[--] The \emph{Siegel} parabolic corresponds to $\{\alpha_1\}$. Letting $t^{\mathrm{Sie}} \defeq e_3^\vee(p) \in T(\Qp)$, we see $v_p(\alpha_1(t^{\mathrm{Sie}})) = 0$ and $v_p(\alpha_2(t^{\mathrm{Sie}})) = -1$, and we get a Siegel-controlling operator $U_p^{\mathrm{Sie}} \defeq U_{t^{\mathrm{Sie}}}$. Let $A_p^{\mathrm{Sie}}$ be the $U_p^{\mathrm{Sie}}$-eigenvalue. Then $\f$ is Siegel-non-critical slope if $v_p(A_p^{\mathrm{Sie}}) < -\langle \lambda, \alpha_2^\vee\rangle\cdot -1 = k_2 + 1$.
			\item[--] The \emph{Klingen} parabolic corresponds to $\{\alpha_2\}$. Letting $t^{\mathrm{Kli}} \defeq (e_2^\vee + 2e_3^\vee)(p) \in T(\Qp)$, we get $v_p(\alpha_1(t^{\mathrm{Kli}})) = -1$ and $v_p(\alpha_2(t^{\mathrm{Kli}})) = 0$; define $U_p^{\mathrm{Kli}} = U_{t^{\mathrm{Kli}}}$, with $\cF$-eigenvalue $A_{p}^{\mathrm{Kli}}$. Then $\f$ is Klingen-non-critical slope if $v_p(A_p^{\mathrm{Kli}}) < -\langle \lambda, \alpha_1^\vee\rangle \cdot -1 = k_1 - k_2 + 1$.
		\end{itemize}
		We may identify the torus in $G$ with a subgroup of the diagonal matrices in $\GL_4$, after which $t^{\mathrm{Sie}}$ is the matrix $\mathrm{diag}(1,1,p,p)$ and $t^{\mathrm{Kli}}$ is $\mathrm{diag}(1,p,p,p^2)$.
	\end{itemize}
\end{examples}

\begin{remark}\label{rem:both non-critical}
	Our definition of $Q$-non-critical uses cohomology with compact support $\hc{\bullet}$; to be more precise, we could call this \emph{$Q$-non-critical for $\hc{\bullet}$}. It is also common to use Betti cohomology (without support) $\h^\bullet$, as in for example \cite{Urb11,Han17}, giving a (directly analogous) notion of \emph{$Q$-non-critical for $\h^\bullet$}. It seems natural to expect that the two notions are equivalent, but it does not \emph{a priori} appear obvious that this is the case. However, Theorem \ref{thm:Q-classicality} applies equally well to both cases: so $Q$-non-critical slope implies both flavours of $Q$-non-criticality. Henceforth, unless specified otherwise, our notion of non-critical should be clear from the underlying setting.
\end{remark}

\subsection{Analytic BGG for the Borel}\label{sec:usual BGG}
We recap the usual locally analytic BGG resolution (Theorem \ref{thm:urban jones}). Recall $\cA(J_B,L)$ is the space of locally $L$-analytic functions on the Iwahori $J_B$, and $\cA_{\lambda}^G \subset \cA(J_B,L)$. We have a left action of $J_B$ on $\cA(J_B,L)$ by $l(h)\cdot f(g) = f(h^{-1}g)$. This action is $L$-analytic, and thus induces an analytic action of $\fg$. Explicitly, $X \in \fg$ acts by
\[
	l(X)\cdot f(g) = \tfrac{d}{dt}\left(\big[l(\exp(-tX))\cdot f\big](g)\right)\big|_{t=0}.
\]
This extends in a natural way to an action of the universal enveloping algebra $\fU(\fg)$.	 

By \cite[Prop.\ 3.2.11]{Urb11}, for each simple root $\alpha \in \Delta$, the map $f \mapsto l(X_\alpha)\cdot f$ induces a map $\cA_{\lambda}^G \to \cA_{\lambda - \alpha}^G$. By \eqref{eq:weyl action on weights}, $w_\alpha * \lambda = \lambda - [\langle \lambda, \alpha^\vee\rangle + 1]\alpha$ for a classical weight $\lambda$, and we have a map
\begin{align*}
\Theta_\alpha: \cA_{\lambda}^G &\longrightarrow \Ara^G\\
f &\longmapsto l\big(X_\alpha^{\langle\lambda, \alpha^\vee\rangle + 1}\big)\cdot f.
\end{align*}
This is $J_B$-equivariant and (recalling $\zeta$ from \eqref{eq:splitting}) transforms under $t \in T^+$ as
\begin{equation}\label{eq:transform}
\Theta_\alpha(t*f) = \alpha(\zeta(t))^{-\langle \lambda, \alpha^\vee\rangle -1}\big[t * \Theta_\alpha(f)\big].
\end{equation}
The following describes the first few terms of the locally analytic BGG resolution. Let $V_{\lambda,\mathrm{loc}}^G \subset \cA_{\lambda}^G$ be the subspace of functions that are locally $L$-algebraic on $J_B$, that is, the union of the subspaces of $s$-algebraic functions over all $s \geq 0$.

\begin{theorem}\emph{\cite[Thm.\ 26]{Jon11}, \cite[Prop.\ 3.2.12]{Urb11}.} \label{thm:urban jones}
	Let $\lambda$ be a classical weight. There is an exact sequence
	\[
		0 \to V_{\lambda,\mathrm{loc}}^G \longrightarrow \cA_{\lambda}^G \xrightarrow{\ \oplus \Theta_\alpha\ } \bigoplus_{\alpha \in \Delta} \Ara^G.
	\]
\end{theorem}

The action of $\fg$ on $\cA(J_B,L)$ preserves $\cA_0(J_B,L)$ (as we can define it on this space directly). Hence we have maps $X_\alpha : \cA_{\lambda,0}^G \to \cA_{\lambda-\alpha,0}^G$ and $\Theta_\alpha: \cA_{\lambda,0}^G \to \cA_{w_\alpha * \lambda,0}^G$.

\begin{corollary} \label{cor:urban jones}
	Let $\lambda$ be a classical weight. There is an exact sequence
	\[
		0 \to V_{\lambda}^G \longrightarrow \cA_{\lambda,0}^G \xrightarrow{\ \oplus \Theta_\alpha\ } \bigoplus_{\alpha \in \Delta} \cA_{w_\alpha * \lambda,0}^G.
	\]
\end{corollary}
\begin{proof}
	Since $V_{\lambda}^G \subset V_{\lambda,\mathrm{loc}}^G$, it is a subset of $\ker(\oplus \Theta_\alpha)$. Conversely, if $f \in \ker(\oplus \Theta_\alpha)$, then by Theorem \ref{thm:urban jones} it lies in $V_{\lambda,\mathrm{loc}}^G \cap \cA_{\lambda,0}^G = V_{\lambda}^G$ (see \cite[\S3.2.9]{Urb11}). 
\end{proof}

\subsection{Theta operators on parahoric distributions}\label{sec:theta operators}
We now describe $\A_{\lambda,0}^Q$ as a canonical subspace of $\cA_{\lambda,0}^G$. If $f \in \A_{\lambda,0}^Q$ and $n \in N_Q(\Zp)$, then by definition $[f(n) : L_Q(\Zp) \to L] \in V_{\lambda}^{L_Q}$.
	 
	\begin{proposition}\label{prop:subspace}
		There is an injective $\Sigma_B$-equivariant map $\iota_Q: \A_{\lambda,0}^Q \hookrightarrow \cA_{\lambda,0}^G$ defined by
		\[
		\iota_Q(f) :J_B \longrightarrow L, \hspace{20pt} \iota_Q(f)(g) \defeq f(g)(\mathrm{id}_{L_Q}).
		\]
	\end{proposition}
\begin{proof}
	Note $\iota_Q(f)$ is analytic since $f$ is.
	Let $t \in T(\Zp)$ and $n^-\in J_B^-$, and write $n^- = n_Q^-\ell^-$ with $n_Q^- \in J_Q^-$ and $\ell^- \in L_Q(\Zp)\cap J_B^-$. Using \eqref{eq:parahoric transform} for $f$ and the $L_Q$-action on $V_{\lambda}^{L_Q}$, we have
	\begin{align*}
		\iota_Q(f)(n^-tg) = f(n_Q^-\ell^-tg)(\mathrm{id}_{L_Q}) &= \left[\langle \ell^-t\rangle_\lambda f(g)\right](\mathrm{id}_{L_Q})\\
		 &= f(g)(\ell^-t) = \lambda(t)f(g)(\mathrm{id}_{L_Q}) = \lambda(t)\iota_Q(f)(g),
	\end{align*}
	so $\iota_Q(f)$ has the right transformation property and $\iota_Q$ is well-defined. Since $f$ is uniquely determined by its restriction to $N_Q(\Zp) \subset J_B$, the map is injective. The map is $J_B$-equivariant since if $j \in J_B$, then $\iota_Q(j*f)(g) = f(gj)(\mathrm{id}_{L_Q}) = \iota_Q(f)(gj) = j*\iota_Q(f)$. If $t \in T^+$, for $b \in B(\Zp)$ similarly $\iota_Q(t*f)(b) = f(\sigma(t)t^{-1}bt)(\mathrm{id}_{L_Q}) = \iota_Q(f)(\sigma(t)t^{-1}bt) = (t*\iota_Q(f))(b)$, so $\iota_Q$ is also $T^+$-equivariant and hence $\Sigma_B$-equivariant.
\end{proof}

From now on, we freely identify $\A_{\lambda,0}^Q$ with its image $\iota_Q(\A_{\lambda,0}^Q)$ in $\cA_{\lambda,0}^G$. We can give an intrinsic criterion for an element of $\cA_{\lambda,0}^G$ to be in this subset.
\begin{definition}
	Let $n \in N_Q(\Zp)$. Define a map
	\[
		\cR_n : \cA_{\lambda,0}^G \longrightarrow \cA_{\lambda,0}^{L_Q},
	\]
 where $\cR_n(f) : L_Q(\Zp) \cap J_B \rightarrow L$ is defined by $\ell \mapsto f(\ell n)$
	(noting that $L_Q(\Zp)\cap J_B$ is the Iwahori subgroup in $L_Q(\Zp)$). Alternatively, $\cR_n(f)$ is the restriction of $(n*f)$ to $L_Q(\Zp)\cap J_B$.
\end{definition}

\begin{proposition}\label{prop:Rn}
	Let $f \in \cA_{\lambda,0}^G$. Then $f \in \A_{\lambda,0}^Q$ if and only if $\cR_n(f) \in V_{\lambda}^{L_Q}$ for all $n \in N_Q(\zp)$, that is, for all $n$ we have 
	\[
	\xymatrix@R=5mm@C=7mm{
		V_{\lambda}^G \sar{r}{\subset}& \A_{\lambda,0}^Q \sar{r}{\subset}\ar[d]^{\cR_n} & \cA_{\lambda,0}^G \ar[d]^{\cR_n}\\
		& V_{\lambda}^{L_Q} \sar{r}{\subset}& \cA_{\lambda,0}^{L_Q}.
	}
	\]
\end{proposition}
\begin{proof}
	If $f' \in \A_{\lambda,0}^Q$, $\ell \in L_Q(\Zp)$ and $n \in N_Q(\Zp)$, then $f'(\ell n)(\mathrm{id}_{L_Q}) = [\langle \ell\rangle_\lambda f'(n)](\mathrm{id}_{L_Q}) = f'(n)(\ell)$. 
		
		Thus if $f = \iota_Q(f')$ for some $f' \in \A_{\lambda,0}^Q$,  then $\cR_n(\iota_Q(f')) = f'(n) \in V_{\lambda}^{L_Q}$. Conversely if $\cR_n(f) \in V_{\lambda}^{L_Q}$ for all $n$, then the function $f' : N_Q(\Zp) \to V_{\lambda}^{L_Q}$ defined by $f'(n) = \cR_n(f)$ defines an element of $\A_{\lambda,0}^{Q}$ and satisfies $\iota_Q(f') = f$.
\end{proof}

\begin{remark}\label{rem:inclusion}
If $P \subset Q$ are two parabolics, as $L_P \subset L_Q$ and $N_Q \subset N_P$ there is an injective extension-by-zero map $V_{\lambda}^{L_P} \hookrightarrow V_{\lambda}^{L_Q}$. We deduce that $\A_{\lambda,0}^Q \subset \A_{\lambda,0}^P$.
\end{remark}
By definition of $\Delta_Q$, if $\alpha \in \Delta_Q$ then $X_\alpha \in \fl_{Q} = \mathrm{Lie}(L_{Q})$, so $\alpha$ is a simple root of $L_{Q}$ and we get a well-defined map $\Theta_\alpha : \cA_{\lambda,0}^{L_Q} \to \cA_{w_\alpha * \lambda,0}^{L_Q}$.

\begin{lemma}\label{lem:theta and R}
	Let $n \in N_Q(\Zp)$. For all $\alpha \in \Delta_Q$, we have a commutative diagram
	\[
		\xymatrix@C=15mm@R=6mm{
				\cA_{\lambda,0}^G \ar[r]^{\Theta_\alpha}\ar[d]^{\cR_n} & \cA_{w_\alpha * \lambda,0}^G\ar[d]\ar[d]^{\cR_n}\\
				\cA_{\lambda,0}^{L_Q} \ar[r]^{\Theta_\alpha} & \cA_{w_\alpha * \lambda,0}^{L_Q}.
			}
	\]
\end{lemma}
\begin{proof}
	It suffices to prove that $\cR_n$ commutes with the action of $X_\alpha$ on $\cA_0(J_B,L)$. But if $f\in \cA_0(J_B,L)$, then for all $\ell \in L_Q(\Zp)\cap J_B$, we have 
	\begin{align*}
		[l(X_\alpha)\cdot \cR_n(f)](\ell) & = \tfrac{d}{dt}\cR_n(f)\big(\exp(-tX_\alpha)\ell\big)|_{t=0}\\
		&= \tfrac{d}{dt}f\big(\exp(-tX_\alpha)\ell n\big)|_{t=0} = [l(X_\alpha)\cdot f](\ell n) = \cR_n(l(X_\alpha)\cdot f)(\ell).\ \ \ \  \qedhere
	\end{align*}
\end{proof}

\begin{lemma}\label{lem:in kernel}
	Suppose $\alpha \in \Delta_Q$. Then  
	$\A_{\lambda,0}^Q \subset \ker(\Theta_\alpha).$
\end{lemma}
\begin{proof}
	If $f \in \A_{\lambda,0}^Q$, then $\cR_n(f) \in V_{\lambda}^{L_Q}$ for all $n \in N_Q(\Zp)$ by Proposition \ref{prop:Rn}; thus 
	\[
	\cR_n(\Theta_\alpha(f)) = \Theta_\alpha(\cR_n(f))\  \in V_{w_\alpha * \lambda}^{L_Q}
	\]
	is also algebraic, the equality being Lemma \ref{lem:theta and R}. Then \ref{prop:Rn} again says $\Theta_\alpha(f) \in \A_{w_\alpha*\lambda,0}^Q$.
	
	As $\alpha$ is a root for $L_{Q}$, the weight $w_\alpha * \lambda$ is not dominant for $L_Q$. It follows that $V_{w_\alpha * \lambda}^{L_Q} = 0$, which forces $\A_{w_\alpha *\lambda,0}^Q = 0$ by definition. It follows that $\A_{\lambda,0}^Q \subset \ker(\Theta_\alpha)$.
\end{proof}

We saw if $\alpha \in \Delta_Q$, then $\Theta_\alpha(\A_{\lambda,0}^Q)\subset\A_{w_\alpha*\lambda,0}^Q$. We want to prove this for $\alpha \notin \Delta_Q$. Such an $\alpha$ is not a root of $L_Q$, so we cannot follow the same strategy. Instead, we argue directly:

\begin{proposition} \label{prop:alpha lands in Q}
	For $\alpha \in \Delta\backslash\Delta_{Q}$, we have $\Theta_{\alpha}\big(\A_{\lambda,0}^{Q}\big) \subset \A_{w_{\alpha}*\lambda,0}^{Q}$. 
\end{proposition}
\begin{proof} 
	Choose a set of co-ordinates $y_i$ on $L_Q(\Zp)\cap J_B$ that identify it as a subset of $\Zp^r$. We also have a set of co-ordinates $z_j$ on $N_Q(\Zp)$, indexed by $j \in R^+\backslash R_Q^+$ as in \eqref{eq:RQ+}. Let $f \in \A_{\lambda,0}^Q$. If $g \in Q(\zp)\cap J_B$, write it as 
	\[
	g = \ell_g n_g, \hspace{20pt} \ell_g \in L_Q(\Zp)\cap J_B, n_g \in N_Q(\Zp).
	\]
	We may write $f(g) = f(y_i(\ell_g), z_j(n_g))$ in the co-ordinates above; then by definition, $f$ is algebraic in the $y_i$ and analytic in the $z_j$.
	
	To show the proposition, by Proposition \ref{prop:Rn} we must show that $\cR_n(\Theta_\alpha(f))$ is algebraic on $L_Q(\zp)\cap J_B$ for all $n \in N_Q(\Zp)$. If $\ell \in L_Q(\zp)\cap J_B$, then
	\[
		\cR_n(\Theta_\alpha(f))(\ell) = (n * \Theta_\alpha(f))(\ell) = \Theta_\alpha(n * f)(\ell),
	\]
	the last equality since $\Theta_\alpha$ respects the $*$ action of $J_B$. Replacing $f$ with $n*f$, it then suffices to prove that the restriction of $\Theta_\alpha(f)$ to $L_Q(\zp)\cap J_B$ lies in $V_\lambda^{L_Q}$. By definition, this is the function 
	\[
		\ell \longmapsto \tfrac{d}{dt}f\big(\mathrm{exp}(-tX_{\alpha})\ell\big)\big|_{t= 0}.
	\]
	Since $\alpha \notin \Delta_Q$, a sufficiently small neighbourhood $U$ of $0$ in $\Qp X_\alpha \subset \fn$ is contained in $\fn_Q(\zp)$. For $t$ in such a $U$, we have $\exp(-tX_\alpha) \in N_Q(\zp)$. This is a normal subgroup in $G(\zp)$, so in particular, for any $\ell \in L_Q(\Zp) \cap J_B$ we have $\exp(-tX_\alpha)\ell = \ell e(\ell,t)$ with $e (\ell,t) = \ell^{-1} \cdot \exp(-tX_\alpha) \cdot \ell \in N_Q(\Zp).$	Then we have
	\begin{align*}
		\Theta(f)(\ell) &= \tfrac{d}{dt} f(\ell e(\ell,t))\big|_{t=0}\\
		&= \tfrac{d}{dt}f(y_i(\ell) ; z_j(e(\ell,t)))|_{t=0}.
	\end{align*}
The co-ordinates $z_j(e(\ell,t))$, which are linear functions in $t$, are algebraic in the $y_i(\ell)$ (since inverse and multiplication operations are algebraic on a reductive group).	We know $f$ is algebraic in the $y_i(\ell)$, and analytic in the $z_j(e(\ell,t))$; and by above the coefficient of the linear term in $t$ is algebraic in the $y_i(\ell)$. We deduce that $\Theta_\alpha(f)(\ell) = \tfrac{d}{dt}f(\ell e(\ell,t))|_{t=0}$ is algebraic in the $y_i(\ell)$, as required.
\end{proof}

\subsection{The parahoric analytic BGG resolution}\label{sec:la BGG}
\begin{proposition}\label{prop:loc an BGG 1}
	For a classical weight $\lambda$, there is an exact sequence
	\[
		0 \to \A_{\lambda,0}^Q \longrightarrow \cA_{\lambda,0}^G \xrightarrow{\oplus \Theta_\alpha} \bigoplus_{\alpha \in  \Delta_Q} \cA_{w_\alpha * \lambda,0}^G.
	\]		
\end{proposition}
\begin{proof}
	That $\A_{\lambda,0}^Q \subset \bigcap \ker(\Theta_\alpha)$ is an immediate consequence of Lemma \ref{lem:in kernel}. To see the converse, suppose $f \in \ker \defeq \bigcap \ker(\Theta_\alpha)$. Then for all $n \in N_Q(\Zp)$, by Lemma \ref{lem:theta and R} we have $\Theta_\alpha(\cR_n(f)) = \cR_n(\Theta_\alpha(f)) = 0$ for any $\alpha \in \Delta_Q$. Thus we have a diagram
	\[
		\xymatrix@C=15mm@R=5mm{
				0 \ar[r] &\ker \ar[r]\ar[d]^{\cR_n} 	& 
				\cA_{\lambda,0}^G \ar[r]^-{\oplus\Theta_\alpha}\ar[d]^{\cR_n} 	& 
				\displaystyle\bigoplus_{\alpha \in \Delta_Q} \cA_{w_\alpha * \lambda,0}^G\ar[d]^{\cR_n}\\
				0 \ar[r] & V_{\lambda}^{L_Q} \ar[r] & \cA_{\lambda,0}^{L_Q} \ar[r]^-{\oplus \Theta_\alpha} &
				\displaystyle\bigoplus_{\alpha \in \Delta_Q} \cA_{w_\alpha* \lambda,0}^{L_Q},
			}
	\]
	where exactness of the bottom row is Corollary \ref{cor:urban jones} for the group $L_Q$, noting that $\Delta_Q$ is precisely the set of simple roots for $L_Q$ corresponding to the Borel $B \cap L_Q$. But then $\cR_n(f) \in V_{\lambda}^{L_Q}$ for any $n$; thus by Proposition \ref{prop:Rn} we have $f \in \A_{\lambda,0}^Q$, as required.
\end{proof}

\begin{corollary}\label{cor:loc an BGG 2}
	Let $P \subset Q$ be two standard parabolics, with $\Delta_P \cup\{\beta\} = \Delta_Q$ (that is, there is no parabolic $P'$ with $P \subsetneq P' \subsetneq Q$). There is an exact sequence
	\[
		0 \to \A_{\lambda,0}^Q \longrightarrow \A_{\lambda,0}^P \xrightarrow{\Theta_{\beta}} \A_{w_\beta * \lambda,0}^P.
	\]
\end{corollary}
\begin{proof}
 	We restrict the map $\oplus\Theta_\alpha$ of \ref{prop:loc an BGG 1} from $\cA_{\lambda,0}^G$ to $\A_{\lambda,0}^P$. It is clear that the kernel of this restriction is $\A_{\lambda,0}^Q \cap \A_{\lambda,0}^P = \A_{\lambda,0}^Q$, the equality following by Remark \ref{rem:inclusion}. If $\alpha \in \Delta_Q$ is not equal to $\beta$, then $\alpha \in \Delta_P$, so $\A_{\lambda,0}^P \subset \ker(\Theta_\alpha)$ by Lemma \ref{lem:in kernel}. In particular, the direct sum $\oplus_{\alpha \in \Delta_Q}\Theta_\alpha$ collapses, with $\Theta_{\beta}$ the only non-zero term. The image lands in $\A_{w_\beta *\lambda,0}^P$ by Proposition \ref{prop:alpha lands in Q}, giving the claimed exact sequence.
 \end{proof}

\subsection{Proof of Theorem \ref{thm:Q-classicality}}
We can finally prove our main result. Recall from Theorem \ref{thm:Q-classicality} that $Q = P_0 \subset \cdots \subset P_m = G$ is a maximal chain of parabolics, $\Delta_{P_{i-1}} \cup \{\alpha_i\} = \Delta_{P_i}$, $t_i \in T^+$ with $v_p(\alpha_i(t_i))<0$, $U_i \defeq U_{t_i}$ and $h_i < h_i^{\mathrm{crit}} \defeq h^{\mathrm{crit}}(t_i,\alpha_i,\lambda)$. 

\begin{proof}\emph{(Theorem \ref{thm:Q-classicality})}.
	First we make sense of taking $U_i$-slope decompositions on $\D^{P_i}$-cohomology. Note that $t_Q = t_1\cdots t_m$ is in $T_Q^{++} \subset T_{P_i}^{++}$ by Proposition \ref{prop:Q controlling}, hence it acts compactly on each $\D_{\lambda,0}^{P_i}$ by Lemma \ref{lem:compact}; we get a $Q$-controlling operator $U_{\mathrm{aux}} = U_{t_Q}$ on $\hc{\bullet}\big(S_K,\D_{\lambda,0}^{P_i}\big)$ for each $i$, and we can take slope decompositions. Let $h_{\mathrm{aux}} \gg v_p(\phi(U_{\mathrm{aux}}))$, so that for each $i$, we have
	\begin{equation}\label{eq:aux slope}
	\hc{\bullet}\big(S_K,\D_{\lambda,0}^{P_i}\big)_\phi \subset  \hc{\bullet}\big(S_K,\D_{\lambda,0}^{P_i}\big)^{U_{\mathrm{aux}} \leq h_{\mathrm{aux}}}.
	\end{equation}

	By the theory of slope decompositions, the right-hand side is Hecke-stable and finite-dimensional over $L$; thus we may take further slope decompositions for $U_i$, as they always exist on finite-dimensional spaces.
	
	\begin{lemma} \label{prop:PQ classicality}
		The map $\rho_\lambda$ induces an isomorphism
		\[
		\big[\h^j_{\mathrm{c}}(S_{K}, \D_{\lambda,0}^{P_{i-1}})^{U_{\mathrm{aux}}\leq h_{\mathrm{aux}}}\big]^{U_i \leq h_i} \isorightarrow \big[\h^j_{\mathrm{c}}(S_{K},\D_{\lambda,0}^{P_{i}})^{U_{\mathrm{aux}}\leq h_{\mathrm{aux}}}\big]^{U_i \leq h_i}.
		\]
	\end{lemma}
	\begin{proof}
		Consider Corollary \ref{cor:loc an BGG 2} applied to the pair $(P_{i-1},P_{i})$. Dualising gives an exact sequence
		\[
		\D_{w_{\alpha_i}\ast \lambda,0}^{P_{i-1}} \xrightarrow{\Theta_{\alpha_i}} \D_{\lambda,0}^{P_{i-1}} \longrightarrow \D_{\lambda,0}^{P_{i}} \to 0.
		\]
		Let $D^{P_{i-1}}= \D_{w_{\alpha_i}\ast \lambda,0}^{P_{i-1}} \big/\mathrm{ker}(\Theta_{\alpha_i})$.  
			We get an induced long exact sequence of cohomology
		\[
		\hc{j}(S_K,D^{P_{i-1}}) \to \hc{j}(S_K,\D_{\lambda,0}^{P_{i-1}}) \to \hc{j}(S_K, \D_{\lambda,0}^{P_{i}}) \to \hc{j+1}(S_K,D^{P_{i-1}}).
		\]
		By \eqref{eq:transform}, this sequence is equivariant for the operators $\alpha_i(\zeta(t_Q))^{-\langle\lambda,H_{\alpha_i}\rangle - 1}U_{\mathrm{aux}}$ (for $D^{P_{i-1}}$-coefficients), and $U_{\mathrm{aux}}$  (for $\D_{\lambda,0}^{P_{i-1}}$, $\D_{\lambda,0}^{P_{i}}$ coefficients); let $h_{\mathrm{aux}}' = h_{\mathrm{aux}} - h^{\mathrm{crit}}(t_Q,\alpha_i,\lambda)$, which is still $\gg v_p(\phi(U_{\mathrm{aux}}))$. As taking slope decompositions is exact \cite[Cor.\ 2.3.5]{Urb11}, after renormalising with \eqref{eq:scaled slope}, for each $i$ we obtain an exact sequence
		\begin{align*}
		\hc{j}\big(S_{K}, D^{P_{i-1}}&\big)^{U_{\mathrm{aux}}\leq h_{\mathrm{aux}}'}
		\rightarrow \hc{j}\big(S_{K},\D_{\lambda,0}^{P_{i-1}}\big)^{U_{\mathrm{aux}}\leq h_{\mathrm{aux}}}\\
		&\rightarrow \hc{j}\big(S_{K}, \D_{\lambda,0}^{P_{i}}\big)^{U_{\mathrm{aux}}\leq h_{\mathrm{aux}}}
		\rightarrow \hc{j+1}\big(S_{K},D^{P_{i-1}}\big)^{U_{\mathrm{aux}}\leq h_{\mathrm{aux}}'}.
		\end{align*}
		This sequence in turn is equivariant for the operators $\alpha_i(\zeta(t_i))^{-\langle \lambda, H_{\alpha_i} \rangle - 1} U_i$ and $U_i$ respectively, and taking further (renormalised) slope decompositions we obtain
		\begin{align*}
		\big[\hc{j}\big(S_{K}, D^{P_{i-1}}&\big)^{U_{\mathrm{aux}}\leq h_{\mathrm{aux}}'}\big]^{U_i \leq h_i- h_i^{\mathrm{crit}}} 
		\rightarrow \big[\hc{j}\big(S_{K},\D_{\lambda,0}^{P_{i-1}}\big)^{U_{\mathrm{aux}}\leq h_{\mathrm{aux}}}\big]^{U_i \leq h_i} \\
		&\rightarrow \big[\hc{j}\big(S_{K}, \D_{\lambda,0}^{P_{i}}\big)^{U_{\mathrm{aux}}\leq h_{\mathrm{aux}}}\big]^{U_i \leq h_i} 
		\rightarrow \big[\hc{j+1}\big(S_{K},D^{P_{i-1}}\big)^{U_{\mathrm{aux}}\leq h_{\mathrm{aux}}'}\big]^{U_i \leq h_i- h_i^{\mathrm{crit}}}.
		\end{align*}	 
		 using that $v_p(\alpha_i(\zeta(t_i))) = v_p(\alpha_i(t_i))$. From \S\ref{sec:integral} and Remark \ref{rem:dot action}, all of the coefficient spaces admit natural $\Sigma_{P_{i}}$-stable integral structures which give natural $U_i$-stable integral structures on the cohomology (and their small slope parts for $U_{\mathrm{aux}}$). As $h_i - h_i^{\mathrm{crit}} < 0$ by assumption, it follows that the first and last terms of the exact sequence vanish by \cite[Lem.\ 9.1]{BW_CJM}.
	\end{proof}
	
	We return to the proof of Theorem \ref{thm:Q-classicality}. For $M$ as above and $\mathbf{h} = (h_1,...,h_m)$, define
	\[
	M^{\leq \mathbf{h}} \defeq \textstyle\bigcap_{i=1}^m \big(M^{U_{\mathrm{aux}}\leq h_{\mathrm{aux}}}\big)^{U_i \leq h_i}.
	\]
	Since $\phi(U_i)$ has $p$-adic valuation $h_i$, we know $U_{i}$ acts with slope $\leq h_i$ on $M_\phi$ for any space $M$; so for each $i$, combining with \eqref{eq:aux slope}, we immediately obtain
	\begin{equation}\label{eq:contained in}
	\hc{\bullet}\big(S_K,\D_{\lambda,0}^{P_i}\big)_\phi \subset  \hc{\bullet}\big(S_K,\D_{\lambda,0}^{P_i}\big)^{\leq \mathbf{h}}.
	\end{equation}
	Thus it suffices to prove that the slope criteria forces
	\[
	\hc{\bullet}\big(S_K,\D_{\lambda,0}^{Q}\big)^{\leq \mathbf{h}} \cong \hc{\bullet}\big(S_K,V_{\lambda}^{G,\vee}\big)^{\leq \mathbf{h}}.
	\]
	For each $i$, using the slope assumption and restricting Lemma \ref{prop:PQ classicality} we obtain isomorphisms 
	\[
	\hc{j}\big(S_K, \D_{\lambda,0}^{P_{i-1}}\big)^{\leq \mathbf{h}} \cong \hc{j}(S_K,\D_{\lambda,0}^{P_{i}})^{\leq \mathbf{h}}.
	\]
	Chaining this together for $i = 1,..., m$, we obtained the claimed isomorphism
	\begin{align*}
	\hc{\bullet}\big(S_K,\D_{\lambda,0}^{Q}\big)^{\leq \mathbf{h}} &= \hc{\bullet}\big(S_K,\D_{\lambda,0}^{P_{0}}\big)^{\leq \mathbf{h}} \cong \cdots \\
	&\cong \hc{\bullet}\big(S_K,\D_{\lambda,0}^{P_{m}}\big)^{\leq \mathbf{h}} = \hc{\bullet}\big(S_K,V_{\lambda}^{G,\vee}\big)^{\leq \mathbf{h}}.\ \ \ \ \qedhere
	\end{align*}
\end{proof}

\subsection{Hecke normalisations and connections to automorphic representations}\label{sec:warning}
We conclude this section with some remarks on applying these results in the context of automorphic representations. Recall from Remark \ref{rem:dot action} that there are two natural actions of $\Sigma_Q$ on $V_{\lambda}^{G,\vee}$: a $*$-action induced by considering $V_{\lambda}^{G,\vee}$ as a stable quotient of $\cD_{\lambda,0}^G$, well-adapted for $p$-adic computations; and a $\cdot$-action coming from the algebraic action, well-adapted to automorphic computations. As explained in Remark \ref{rem:dot action} these actions agree on $K$, so give the same local system on $S_K$; but they \emph{differ} on $T^+$, giving different Hecke actions on the resulting cohomology. Thus attached to $t \in T^+$ we get two Hecke operators $U_t^*$ and $U_t^\cdot$ on $\hc{\bullet}(S_K,V_{\lambda}^{G,\vee})$, and by \eqref{eq:modify}, we have
\[
	U_t^* = \lambda(\sigma(t)t^{-1}) \times U_t^{\cdot}.
\]

Now, let $\pi$ be a cuspidal cohomological automorphic representation of $\cG(\A)$ admitting $K$-invariant vectors, and fix an eigenform $\cF\in \pi_f^K$. We call the pair $(\pi,\cF)$ a \emph{$p$-refinement of $\pi$} and denote it $\tilde\pi$, with associated eigensystem $\phi_{\tilde\pi}^\cdot : \uhp_K\to \C$. In favourable situations, one may use Lie algebra cohomology and complex periods to construct a (typically non-canonical) Hecke eigenclass $\psi_{\tilde\pi} \in \hc{\bullet}(S_K,V_{\lambda}^{G,\vee})$ which \emph{for the $\cdot$-action} has the same Hecke eigenvalues as $\tilde\pi$. In particular, if we view $\psi_{\tilde\pi}$ as a class in the cohomology with the $*$-action -- as we have done throughout this paper -- we must instead consider slope conditions for $\phi_{\tilde\pi}^*(U_t) = \lambda(\sigma(t)t^{-1})\times\phi_{\tilde\pi}^\cdot(U_t)$. We summarise this in the following corollary of Theorem \ref{thm:Q-classicality}.

\begin{corollary}
	Let $\tilde\pi$ be as above. Let $Q = P_0 \subset \cdots \subset P_m = G$ be a maximal chain of parabolics, and for $i = 1,...,m$, let $U_i$ be as in Theorem \ref{thm:Q-classicality}. Let $a_i  = \phi_{\tilde\pi}^\cdot(U_i)$ denote the $U_i$ eigenvalue of $\cF$, and let $a_i^\circ \defeq \lambda(\sigma(t)t_i^{-1})a_i$, an `integral normalisation' of $a_i$. If
	\[
		v_p(a_i^\circ) = v_p(\lambda^{-1}(t_i)) + v_p(a_i) < h^{\mathrm{crit}}(t_i,\alpha_i,\lambda)
	\]
	for all $i = 1,...,m$, then $\tilde\pi$ is $Q$-non-critical.
\end{corollary}
(Note the modification factor is only $v_p(\lambda^{-1}(t_i))$ as $v_p(\lambda(\sigma(t_i))) = 0$). Sometimes the operator $\lambda(\sigma(t)t^{-1})U_t^\cdot$ is denoted $U_t^\circ$, and (in light of Remark \ref{rem:dot action}) is called the `optimal integral normalisation' of the classical automorphic Hecke operator $U_t^\cdot$. In other places -- e.g.\ \cite{Han17} -- the specialisation map is defined using the $*$-action on distributions and $\cdot$-action on algebraic coefficients, and is then referred to as an `intertwining' of $U_t^*$ and $\lambda(\sigma(t)t^{-1})U_t^\cdot$.

\begin{remark}
We remark finally that it there are two common sets of conventions when defining local systems. We have taken all of our modules to be \emph{right} $K$-modules, as this is natural/standard in the $p$-adic setting. For automorphic computations is is frequently more natural to consider only \emph{left} $K$-modules. This then flips every convention in this paper, so that for example the $\leq$ and $<$ of Propositions \ref{prop:valuation T++} and \ref{prop:Q controlling} become $\geq$ and $>$, the action of $K$ is induced by right-translation by the inverse, and the $*$-action is induced by $g \mapsto \sigma(t)^{-1}tgt^{-1}$. In particular, a controlling operator for $\GL_2$ would be given by $\smallmatrd{p}{}{}{1}$ rather than $\smallmatrd{1}{}{}{p}$. Let $w_0$ be the longest Weyl element for $G$, and let $\lambda^\vee = -\lambda^{w_0}$ denote the contragredient weight. Since $V_{\lambda}^{G,\vee} \cong V_{\lambda^\vee}^G$ when equipped with the left $\cdot$-actions, by mimicking the calculation of Remark \ref{rem:dot action}, we see that in this set-up we have instead that $U_t^* = \lambda^\vee(\sigma(t)^{-1}t) \times U_t^\cdot$, and we would define $a_i^\circ = \lambda^\vee(\sigma(t)^{-1}t_i)a_i$.
 \end{remark}

%%==============================================
%
%			PARABOLIC EIGENVARIETIES
%
%%==============================================
\section{Parabolic eigenvarieties}\label{sec:parabolic eigenvarieties}

We now construct a theory of parabolic families of automorphic representations. There are two approaches to constructing eigenvarieties from overconvergent cohomology, with differing benefits and drawbacks. We could use \emph{total} cohomology, as in \cite{Urb11, Han17}, giving more accessible general results; or a single degree of cohomology, which is often of more arithmetic use (see, for example, the `middle degree' eigenvariety of \cite{BH17}, or the `parallel weight' eigenvariety of \cite{BW18}). This, however, requires the pinning down of Hecke eigenpackets in the specified degree, so typically requires more refined arguments to study. In the below, a $*$ will denote either total cohomology $\bullet$ or a specific degree $d \in \Z_{\geq 0}$.

Fix throughout a parabolic subgroup $Q$, and a level group $K$ with $K_p \subset J_Q$; all our eigenvarieties will depend on this $K$, but since it is fixed we drop it from all notation.  Fix also a `base-weight' $\lambda_0 \in \cW$, giving a subspace $\Wlam \subset \W$ as in \S\ref{sec:weight spaces}. All other notation will be as above.

\subsection{Local pieces of the eigenvariety}
The eigenvarieties we consider are defined using the parahoric overconvergent cohomology groups for $Q$. The local pieces are defined as the rigid analytic spectra of Hecke algebras acting on these spaces.

Fix for the rest of this section a controlling operator $U_t$ (for $t \in T^{++}_Q$); all slope decompositions will be with respect to $U_t$. Let $\cU\subset \Wlam$ be an affinoid. The pair $(\cU,h)$ is a \emph{slope-adapted pair} if the cohomology $\h^*_{\mathrm{c}}(S_K,\D_{\cU}^Q)$ admits a slope $\leq h$ decomposition. Recall $\uhp(K)$ from Definition \ref{def:app hecke algebra}.

\begin{definition} 
	For a slope-adapted pair $(\cU,h)$, let 
	\[
	\bT_{\cU,h}^{Q,*} \defeq \text{ image of }\cH(K)\otimes_{\Qp}\cO(\cU)\text{ in }\mathrm{End}_{\cO(\cU)}\big(\h_{\mathrm{c}}^*\big(S_K,\D_{\cU}^Q\big)\ssh\big).
	\]
	Define the \emph{local piece of the eigenvariety} at $(\cU,h)$ to be the rigid analytic space
	\[
	\cE_{\cU,h}^{Q,*} \defeq \mathrm{Sp}\big(\bT_{\cU, h}^{Q,*}\big).
	\]
\end{definition}

The natural structure map $\cO(\cU) \rightarrow \bT_{\cU,h}^{Q,*}$ gives rise to a map $w : \cE_{\cU,h}^{Q,*} \rightarrow \cU$, which we call the \emph{weight map}. We get the following key property essentially by definition.

\begin{proposition}\label{prop:local eigenvariety} There is a bijection between:
	\begin{itemize}\setlength{\itemsep}{0pt}
		\item $L$-points $x = x(\phi)$ of the rigid space $\cE_{\cU,h}^{Q,*}$ with $w(x) = \lambda$, and
		\item systems of Hecke eigenvalues $\phi : \uhp(K) \to L$ that occur in the localisation 
		\[
		\hc{*}\big(S_K,\D_{\cU}^Q\big)_\lambda^{\leq h} \defeq \hc{*}\big(S_K,\D_{\cU}^Q\big)^{\leq h} \otimes_{\cO(\cU)} \cO(\cU)_{\m_\lambda},
		\]
		where $\m_\lambda \subset \cO(\cU)$ is the maximal ideal corresponding to $\lambda$.
	\end{itemize}
\end{proposition}
\begin{proof}
	Such a point $x$ corresponds to a maximal ideal in $\m_x \subset \bT_{\cU,h}^{Q,*}$ with residue field $L$, and we obtain a surjective algebra homomorphism
	\[
		\phi_x : \uhp(K) \twoheadrightarrow \bT_{\cU,h}^{Q,*} \twoheadrightarrow \bT_{\cU,h}^{Q,*}/\m_x \cong L,
	\]
	which by definition occurs in $\hc{*}(S_K,\D_{\cU}^Q)^{\leq h}$. To say that $w(x) = \lambda$ means that the contraction $\m_x \cap \cO(\cU) = \m_\lambda$, and thus $\phi_x$ occurs in the stated localisation.
\end{proof}

\subsection{The global eigenvariety}\label{sec:global eigenvariety}
These local pieces glue into a `global' eigenvariety over the weight space $\Wlam$. This is straightforward using the `eigenvariety machine' of \cite[\S4.2]{Han17}; although non-minimal parabolics do not feature in Hansen's paper, the formalism of this machine carries over to this case with little (and often no) modification. As such, our treatment of the material will be terse. The key will be to identify an \emph{eigenvariety datum}, from which we may apply Theorem 4.2.2 \emph{op.\ cit}.\ to obtain our global eigenvariety.

\subsubsection{Fredholm power series and hypersurfaces}

The modules of analytic functions from previous sections give rise to Borel--Moore chain complexes $C_*^{\mathrm{BM}}(K,\A_{\cU,s}^Q)$ (dual to the compactly supported complex with distributions defined previously). The proofs of Propositions 3.1.2--3.1.5 of \cite{Han17} hold in our setting with no modification, showing that the (small slope) homology and cohomology of these complexes is compatible with changing the affinoid $\cU$. 

For each affinoid open $\cU \subset \Wlam$, and $s \geq s[\cU]$, we define a Fredholm series
\[
F_{\cU,s}^{Q}(X) \defeq \det\left(1-U_tX|C_*^{\mathrm{BM}}(K,\A_{\cU,s}^Q)\right) \in \cO(\cU)[[X]].
\]
This is independent of the choice of $s \geq s[\cU]$ (as in \cite[Proposition 3.1.1]{Han17}), so we simply denote it $F_{\cU}^Q(X)$. By Tate's acyclity theorem, there exists a unique 
\[
F^Q(X) \in \cO(\Wlam)[[X]]
\]
such that $F^Q(X)|_{\cU} = F_{\cU}^Q(X)$, and this is a Fredholm series over all of $\Wlam$. In particular, this defines a Fredholm hypersurface $\sZ^Q \subset \Wlam \times \mathbf{A}^1$, where $\mathbf{A}^1$ is affine 1-space. There is a natural map $w: \sZ^Q \rightarrow \Wlam$ given by projection to the first factor, and this has open image (see \cite[Proposition 4.1.3]{Han17}).

\begin{proposition}\emph{\cite[Prop.\ 4.1.4 and preceding discussion]{Han17}, or \cite[\S4]{Buz07}.}
	Let 
	\[
	\sZ_{\cU,h}^Q = \mathrm{Sp}\big(\cO(\cU)\langle p^hX\rangle/(F_{\cU}^Q(X))\big) \subset \sZ^Q.
	\]
	The natural map $\sZ_{\cU,h}^Q \rightarrow \cU$ is finite flat if and only if $(\cU,h)$ is a slope-adapted pair; we call such $\sZ_{\cU,h}^Q$ a \emph{slope-adapted affinoid}. The set of slope-adapted affinoids is an admissible cover of $\sZ^Q$.
\end{proposition}

\subsubsection{The eigenvariety datum} 
We use the above to define an eigenvariety datum giving rise to the parabolic eigenvarieties. Indeed, the proof of \cite[Proposition 4.3.1]{Han17} (and the following paragraph) shows that there is a unique coherent analytic sheaf $\sM^*$ on $\sZ^Q$ such that 
\[
\sM^*(\sZ_{\cU,h}^{Q}) = \hc{*}\big(S_K,\D_{\cU}^Q\big)^{\leq h}.
\]
We then let $\psi: \cH(K) \rightarrow \mathrm{End}_{\cO_{\sZ^Q}}(\sM^*)$ be the obvious map giving the Hecke action on cohomology. Then $(\Wlam,\sZ^Q,\sM^*,\cH(K),\psi)$ is an eigenvariety datum, and \cite[Theorem 4.2.2]{Han17} then allows us to glue the local pieces of Proposition \ref{prop:local eigenvariety} into the following:

\begin{theorem}
	There exists a separated rigid analytic space $\cE_{\lambda_0}^{Q,*}$, together with a morphism $w : \cE_{\lambda_0}^{Q,*} \rightarrow \Wlam$, such that for every finite extension $L$ of $\Qp$, there is a bijection between:
	\begin{itemize}\setlength{\itemsep}{0pt}
		\item the $L$-points $x = x(\phi)$ of $\cE_{\lambda_0}^{Q,*}$ lying above a weight $\lambda$, and
		\item systems of Hecke eigenvalues $\phi : \uhp(K) \to L$ with $\phi(U_t) \neq 0$ that occur in the localisation 
		$
		\hc{*}\big(S_K,\D_{\cU}^Q\big)_\lambda,
		$
		where $\cU$ is any sufficiently small open affinoid containing $\lambda$.	
	\end{itemize}
\end{theorem}

\begin{remark}\label{rem:coherent}
Additionally, during the construction one obtains a canonical coherent sheaf $\sM$ on $\cE_{\lambda_0}^{Q,*}$ that interpolates $Q$-finite slope eigenspaces in the spaces $\hc{*}(S_K,\D_\cU^Q)$.
 \end{remark}

\begin{remark}\label{rem:tame level}
	If $K = K^pK_p$ with $K^p \subset \cG(\A_f^{(p)})$, then the eigenvariety of level $K$ depends only on the `tame level at $Q$', that is, $K^p$ and $K_p\cap L_Q(\Zp)$. Indeed, let $t \in T_Q^{++}$ such that $v_p(\alpha(t)) = 0$ for all $\alpha \in \Delta_Q$. This condition ensures conjugation by $t^{-1}$ preserves $L_Q(\Zp)$, whilst by Proposition \ref{prop:Q controlling}, it increases (resp.\ decreases) the $p$-adic valuation of all non-trivial entries of $n^- \in N_Q^-(\Zp)$ (resp.\ $n \in N_Q(\Zp)$). Thus if $K'_p \subset K_p$ with $K'_p \cap L_Q(\Zp) = K_p \cap L_Q(\Zp)$,  then necessarily there exists an integer $r$ such that 
	\begin{equation}\label{eq:smaller level}
		t^{r}K_p t^{-r} \cap K_p \subset K_p'.
	\end{equation}
	Let $K' = K^pK_p'$, and $\mathrm{res}_{K'}^{K}$ denote restriction from level $K$ to $K'$ on cohomology. Then \eqref{eq:smaller level} ensures that $U_t^r = [Kt^rK] = [K't^r K] \circ \mathrm{res}_{K'}^K$, and $(U_t')^r \defeq [K't^rK'] = \mathrm{res}_{K'}^K\circ [K't^r K]$. In particular, the $Q$-finite slope eigensystems at level $K$ and $K'$ are the same, and the eigenvarieties at level $K$ and $K'$ are the same. (In fact, this is further true if $K$ and $K'$ have the same intersection with $\cG(\A_f^{(p)}) \times L_Q(\Zp)^{\mathrm{der}}$, since the centre acts trivially on the coefficient modules).
\end{remark}

\subsection{Hansen's spectral sequences}
To study the geometry of these local pieces, we use spectral sequences introduced by Hansen in \cite[\S3.3]{Han17}. From the construction, it can be seen that the only input required to define these spectral sequences is a theory of distributions that behaves well with respect to base-change of the weight. But all of the foundational results for fully overconvergent distributions given and used in \cite{Han17} hold for parahoric distributions via exactly the same proofs, so the construction of the spectral sequences also carries over, and we conclude:
\begin{proposition}\label{prop:spectral}
	Fix a slope adapted pair $(\cU,h)$ with $\cU \subset \Wlam$, and let $\Sigma \subset \cU$ be an arbitrary Zariski-closed subspace. Then $\hc{\bullet}(S_K,\D_{\Sigma}^Q)$ admits a slope $\leq h$ decomposition, and there is a convergent first quadrant spectral sequence
	\[
	E_2^{i,j} = \mathrm{Ext}^i_{\cO(\cU)}\left(\h_j^{\mathrm{BM}}\big(S_K,\A_{\cU}^Q\big)^{\leq h}, \cO(\Sigma)\right) \implies \hc{i+j}\big(S_K,\D_{\Sigma}^Q\big)^{\leq h},
	\]
	and a convergent second quadrant spectral sequence
	\[
	E_{2}^{i,j} = \mathrm{Tor}_{-i}^{\cO(\cU)}\left(\hc{j}\big(S_K,\D_{\cU}^Q\big)^{\leq h}, \cO(\Sigma)\right) \implies \hc{i+j}\big(S_K,\D_{\Sigma}^Q\big)^{\leq h}.
	\]
	There are analogous spectral sequences replacing compactly supported cohomology and Borel--Moore homology with singular cohomology and singular homology, and -- considering the Borel--Serre compactification of $S_K$ -- with boundary cohomology and boundary homology. These spectral sequences are all equivariant for the action of $\uhp(K)$ on their $E_2$ pages and abutments.
\end{proposition}

Many of the general consequences Hansen and Newton obtain from studying these spectral sequences also carry over, with identical proofs, to the context of parabolic eigenvarieties. We highlight some of these, referencing only the relevant equivalent statement in \cite{Han17}.

\subsubsection{Points in eigenvarieties and classical eigensystems}

\begin{proposition}\label{prop:U to lambda}\cite[Thm.\ 4.3.3]{Han17}.
	For total cohomology $* = \bullet$, there is a bijection between:
	\begin{itemize}\setlength{\itemsep}{0pt}
		\item $L$-points $x = x(\phi)$ of $\cE_{\lambda_0}^{Q,\bullet}$ lying above a weight $\lambda \in \Wlam$, and
		\item systems of Hecke eigenvalues $\phi : \uhp(K) \to L$ with $\phi(U_t) \neq 0$ that occur in $\hc{\bullet}(S_K,\D_{\lambda}^Q(L))$.
	\end{itemize}
	Thus if $\tilde\pi$ is a $Q$-non-critical $p$-refined automorphic representation of weight $\lambda \in \Wlam$, with Hecke eigensystem $\phi_{\tilde\pi}$ with $\phi_{\tilde\pi}(U_t) \leq h$, then there is a point $x_{\tilde\pi}$ in $\cE_{\lambda_0}^{Q,\bullet}$ corresponding to $\tilde\pi$.
\end{proposition}
In particular, the coherent sheaf $\sM$ of Remark \ref{rem:coherent} actually interpolates $Q$-finite slope eigenspaces in $\hc{\bullet}(S_K,\D_\lambda^Q)$, and hence -- via Theorem \ref{thm:Q-classicality} -- $Q$-non-critical eigenspaces in classical cohomology. We also have a partial analogue of this when we work in a single degree $* = d\in \Z_{\geq 0}$:

	\begin{proposition} \label{prop:U to lambda d} 
		If $x = x(\phi) \in \cE_{\lambda_0}^{Q,d}(L)$ lies above $\lambda \in \cW_{\lambda_0}^Q$, then $\phi$ occurs in $\hc{d}(S_K,\D_{\lambda}^Q(L))$.
	\end{proposition}
	\begin{proof}
		Let $\cU$ be a neighbourhood of $\lambda$. By Proposition \ref{prop:local eigenvariety}, $\phi$ occurs in $\hc{d}(S_K,\D_{\cU}^Q)_\lambda$. We localise the Tor spectral sequence (with $\Sigma = \lambda$) at $\phi$. Since $\phi(U_t) \neq 0$, reducing modulo $\m_{\lambda}$ we see $\phi$ occurs in $E_2^{0,d} = \hc{d}(S_K,\D_{\cU}^Q)\otimes_{\cO(\cU)}\cO(\cU)/\m_{\lambda}.$ Hence it occurs in the abutment $E_\infty^{0,d} = \hc{d}(S_K,\D_{\lambda}^Q(L))$.
	\end{proof}

\begin{remark}
	The converse to Proposition \ref{prop:U to lambda d} is \emph{false}. For example, let $\cG = \GL_2/\Q$, $Q=B$, and $K = K_0(p)$. Then as explained in \cite[Thm.\ 7.1]{PS12} and \cite[Thm.\ 3.30]{Bel12}, there is a weight 2 critical Eisenstein series $E_2^{\mathrm{crit}}$ whose eigensystem $\epsilon$ occurs in $\hc{1}(S_K,\D_{(0,0)}^B)$ but not in $M_2^\dagger(K)$, hence not in the Coleman--Mazur eigencurve (denoted $\cE_{(0,0)}^{B,1}$ here). In particular $\epsilon$ does not occur in $\hc{1}(S_K,\D_{\cU}^B)$. It does appear in $\hc{2}(S_K,\D_{\cU}^B)$ and hence $\cE_{(0,0)}^{B,\bullet}$, consistent with Proposition \ref{prop:U to lambda}.
\end{remark}

\subsubsection{The dimension of components}
 Let $\tilde\pi$ be a $p$-refined automorphic representation and suppose there is an attached point $x_{\tilde\pi} \in \cE_{\lambda_0}^{Q,*}$. A \emph{$p$-adic family through $\tilde\pi$} is a positive-dimensional component of $\cE_{\lambda_0}^{Q,*}$ through $x_{\tilde\pi}$. It is not obvious that such a component exists. 
In the case of total cohomology, however, we can exhibit lower bounds on the dimensions of components.

\begin{definition}\label{def:defect}
	Let $x = x(\phi)$ be a point of $\cE_{\lambda_0}^{Q,\bullet}$, and let $t_Q(x)$ (resp.\ $b_Q(x)$) denote the supremum (resp. infemum) of the set $\{ i \in \Z : \hc{i}(S_K,\D^Q_\lambda(L))_\phi \neq 0\}$. Define the \emph{overconvergent defect at $Q$} to be $\ell_Q(x) = t_Q(x) - b_Q(x)$.
\end{definition}

The following proposition is proved exactly as in Newton's proof of \cite[Prop.\ B.1]{Han17}.

\begin{proposition}\label{prop:dimension}
	Any irreducible component of $\cE_{\lambda_0}^{Q,\bullet}$ containing a given point $x$ has dimension at least $\dim \cW_{\lambda_0}^Q - \ell_Q(x)$.
\end{proposition}

\subsection{Parabolic families of cuspidal automorphic representations}\label{sec:cuspidal families}
We now consider the construction of $p$-adic families of cuspidal automorphic representations. Let $x = x(\phi) \in \cE_{\lambda_0}^{Q,*}(L)$ be a point of classical (dominant) weight $\lambda = w(x)$, corresponding to a system of eigenvalues $\phi : \cH(K)\to L$.  We require some further technical conditions:

\begin{definition}
	\begin{enumerate}\setlength{\itemsep}{0pt}
		\item We say $\phi$ is \emph{interior} if $\h^\bullet_{\partial}(S_K,V_\lambda^\vee(L))_\phi =0$ (boundary cohomology for the Borel--Serre compactification of $S_K$), and \emph{$Q$-strongly interior} if $\h^\bullet_\partial(S_K,\D_\lambda^Q(L))_\phi = 0$. 
		\item We say $\lambda$ is \emph{regular} if  $\langle \lambda, \alpha\rangle > 0$ for all simple roots $\alpha$ of $G$. (For $\GL_n$, this means that $\lambda = (\lambda_1,...,\lambda_n)$ with $\lambda_i > \lambda_{i+1}$ for all $i$).
		\item We say $x$ is a \emph{classical cuspidal point} if there exists a cuspidal automorphic representation $\pi_x$ such that $\phi$ occurs in $\pi_x^K$ (after appropriate renormalisations as in \S\ref{sec:warning}).
	\end{enumerate}
\end{definition}

\begin{lemma}\label{lem:strongly interior}
	If $\phi$ is $Q$-non-critical slope, then it is $Q$-strongly interior if and only if it is interior. This is true more generally if $\phi$ is $Q$-non-critical for both $\hc{\bullet}$ and $\h^\bullet$ (Remark \ref{rem:both non-critical}).
\end{lemma}
\begin{proof}
	That the second statement is a generalisation of the first is Remark \ref{rem:both non-critical}. To prove the second, combining specialisation $\rho_\lambda$ with the excision exact sequence,
	we have a commutative and Hecke-equivariant diagram with exact rows, which we localise at $\phi$:
	\[
	\xymatrix@R=7mm{
		\cdots \ar[r] & \h^{\bullet}_{\mathrm{c}}(S_K,\D^Q_\lambda(L))_\phi \ar[r]^{\iota_1}\ar[d]^{\rho_\lambda}  & \h^{\bullet}(S_K,\D^Q_\lambda(L))_\phi \ar[r]\ar[d]^{\rho_\lambda} &\h^{\bullet}_\partial(S_K,\D_\lambda^Q(L))_\phi \ar[r]\ar[d]^{\rho_\lambda} & \cdots\\
		\cdots \ar[r] & \h^{\bullet}_{\mathrm{c}}(S_K,V_\lambda^\vee(L))_\phi \ar[r]^{\iota_2}  & \h^{\bullet}(S_K,V_\lambda^\vee(L))_\phi \ar[r] &\h^{\bullet}_\partial(S_K,V_\lambda^\vee(L))_\phi \ar[r] & \cdots
	}
	\]
	Being $Q$-strongly interior (resp.\ interior) is equivalent to $\iota_1$ (resp.\ $\iota_2$) being an isomorphism in every degree. 
	By $Q$-non-criticality, the first and second vertical maps are isomorphisms; thus $\iota_1$ is an isomorphism in every degree if and only if $\iota_2$ is, from which we conclude. 
\end{proof}

Let $x = x(\phi) \in \cE_{\lambda_0}^{Q,*}(L)$ be as above, of weight $\lambda = w(x)$, and suppose that:
\begin{itemize}\setlength{\itemsep}{0pt}
	\item[(C1)] $x$ is a classical cuspidal point,
	\item[(C2)] $\phi$ is $Q$-non-critical for $\hc{\bullet}$ and $\h^{\bullet}$,
	\item[(C3)] and that every neighbourhood of $\lambda$ in $\cW^Q_{\lambda_0}$ contains a Zariski-dense set of regular weights. (Note this is automatic if $\lambda$ itself is a regular weight).
\end{itemize}

\begin{proposition}\label{prop:zariski dense cuspidal}
Let $\cV$ be an irreducible component of $\cE_{\lambda_0}^{Q,*}$ passing through $x$. If $\dim\cV = \dim \cW_{\lambda_0}^Q,$ then $\cV$ contains a Zariski-dense set $\cV^{\mathrm{cl}}$ of classical cuspidal (cohomological) points.
\end{proposition}
\begin{proof}
	If an open neighbourhood of $x$ in $\cV$ contains a Zariski-dense set of such points, then $\cV$ does. Thus we work locally, and assume $\cV$ is a component of a local piece $\cE_{\cU,h}^{Q,*}$ containing $x$. 
			
	By (C3), we may always pick a Zariski-dense subset $\cU^{\mathrm{cl}} \subset \cU$ of classical (regular) weights $\lambda'$ for which $h$ is a $Q$-non-critical slope. Let $\cV^{\mathrm{cl}}$ denote the set of $y \in \cE_{\cU,h}^{Q,*}$ with $\lambda_y \defeq w(y) \in \cU^{\mathrm{cl}}$; this set is necessarily Zariski-dense in $\cV$ as $\dim(\cU) = \dim(\cV)$. If $y \in \cV^{\mathrm{cl}}$, then by Propositions \ref{prop:U to lambda} (for $* = \bullet$) or \ref{prop:U to lambda d} (for $* = d$) it corresponds to a system of eigenvalues $\phi_y$ occurring in 
	\[
	\hc{*}(S_K,\D_{\lambda_y}^Q(L))^{\leq h} \cong \hc{*}(S_K,V_{\lambda_y}^\vee(L))^{\leq h},
	\]
	isomorphic since $h$ is a $Q$-non-critical slope for $\lambda_y$. Hence $\phi_y$ is a system of eigenvalues in the classical cohomology. It remains to show $\phi_y$  is cuspidal (i.e.\ occurs in the cuspidal cohomology).
	
	Since $\tilde\pi$ is cuspidal, then by a theorem of Borel (see for example \cite[\S0]{LS04}) the associated eigensystem $\phi$ is interior; thus it is $Q$-strongly interior by Lemma \ref{lem:strongly interior}.	
	 Analogously to \cite[Theorem 4.5.1(ii)]{Han17}, from the boundary Tor spectral sequence (Proposition \ref{prop:spectral}) we deduce that $\h^\bullet_\partial(S_K,\D_{\cU}^Q)_\phi = 0$. The boundary cohomology yields a coherent sheaf on the eigenvariety, and we see that the rigid localisation of this sheaf at $x$ -- which is a faithfully flat extension of the algebraic localisation -- must be zero. Thus, perhaps after shrinking the neighbourhoods $\cU$ and $\cV$, this vanishing lifts to $\cV$. Thus for any $y \in \cV$, we have $\h^\bullet_\partial(S_K,\D_{\cU}^Q)_{\phi_y} = 0$; and localising the boundary spectral sequence at $y$, we deduce that $\h^{\bullet}_\partial(S_K,\D_{\lambda_y}^Q(L))_{\phi_y} = 0$ and $\phi_y$ is strongly interior.
	
	Now suppose $y \in\cV^{\mathrm{cl}}$. Since $\phi_y$ is $Q$-non-critical slope, by Lemma \ref{lem:strongly interior} it is interior. But for regular weights, a class is interior if and only if it is cuspidal \cite[Prop. 5.2, \S5.3]{LS04}, so $\phi_y$ appears in the cuspidal cohomology, as required.
\end{proof}

As a special case where the dimension hypothesis on $\cV$ will always be satisfied, we have:

\begin{corollary}
	Suppose $\cG^{\mathrm{der}}(\R)$ admits discrete series, and let $x  \in \cE_{\lambda_0}^{Q,*}$ satisfy (C1-3). Every irreducible component of $\cE_{\lambda_0}^{Q,*}$ through $x$ contains a Zariski-dense set of classical cuspidal points.
\end{corollary}
\begin{proof}
	The conditions on $\cG$ and $x= x(\phi)$ ensure that $\phi$ appears in only one degree of classical cohomology (e.g.\ \cite[\S4-5]{LS04}); and then Proposition \ref{prop:dimension} ensures that any such irreducible component has dimension $\dim \cW_{\lambda_0}^Q$. We conclude by Proposition \ref{prop:zariski dense cuspidal}.
\end{proof}

\begin{remark}
	The assumptions on regular weights ensure control over the \emph{classical} cohomology, and in situations where we have a more complete understanding of the classical cohomology -- for example, the case of $\GL_2$ -- we may relax these conditions. 
	
	For $B$-families, every affinoid neighbourhood of a classical weight $\lambda_0$ contains a Zariski-dense set of regular classical weights. If $\lambda_0$ is not regular, this is not necessarily true in the parahoric case. For example, consider $\cG= \GL_4$, and $\lambda_0 = (0,0,0,0)$, and $Q$ with Levi $\GL_2 \times \GL_2$. Then every weight $\lambda = (\lambda_1,...,\lambda_4) \in \cW^Q_{\lambda_0}$ has $\lambda_1 = \lambda_2$ and $\lambda_3 = \lambda_4$, so this space contains \emph{no} regular weights.
\end{remark}

\begin{remark}
	Suppose $\cG^{\mathrm{der}}(\R)$ does not admit discrete series. Then if a point $x$ is cuspidal $Q$-non-critical, then $\ell_Q(x) \geq 1$. When $Q = B$, \cite[Thm.\ 4.5.1]{Han17} says that irreducible components through such $x$ \emph{never} have maximal dimension (that is, dimension equal to $\dim\cW$), and conjecturally the inequality of Proposition \ref{prop:dimension} is an equality. This conjecture is false in the general parahoric setting. Indeed, in \cite{BDW20} examples are given of $Q$-parabolic families of dimension $\mathrm{dim}\cW_{\lambda_0}^Q$ in the setting of $\cG = \mathrm{Res}_{F/\Q}\GL_{2n}$, even though $\ell_Q(x) = n-1$. Conceptually these families arise through transfer from $\mathrm{GSpin}_{2n+1}$ (where we \emph{do} have discrete series).
\end{remark}

\footnotesize
\renewcommand{\baselinestretch}{1}
\renewcommand{\refname}{\normalsize References} 
\bibliography{master_references}{}
\bibliographystyle{alpha}
\Addressesshort

\end{document}